\documentclass[reqno]{amsart}

\usepackage{tikz-cd}

\usepackage{fullpage,dsfont}

\usepackage{hyperref} % this should come before amsrefs

\usepackage{parskip}
\makeatletter % need this to avoid the conflict between amsthm and parskip
\def\thm@space@setup{%
 \thm@preskip=\parskip \thm@postskip=0pt
}
\def\th@remark{%
  \thm@headfont{\itshape}%
  \normalfont % body font
  \thm@preskip\parskip \thm@postskip=0pt
}
\makeatother

\usepackage[nobysame,alphabetic,initials,msc-links]{amsrefs}

\usepackage[final]{pdfpages}

\usepackage{multirow}

\usepackage{array}

\usepackage{tikz-cd}
\usetikzlibrary{knots}
\usepackage{spath3}

\usepackage{float}

\DefineSimpleKey{bib}{how}
\DefineSimpleKey{bib}{mrclass}
\DefineSimpleKey{bib}{mrnumber}
\DefineSimpleKey{bib}{fjournal}
\DefineSimpleKey{bib}{mrreviewer}

\renewcommand{\PrintDOI}[1]{%
  \href{http://dx.doi.org/#1}{{\tt DOI:#1}}%
}
\renewcommand{\eprint}[1]{#1}
\BibSpec{book}{%
    +{}  {\PrintPrimary}                {transition}
    +{.} { \PrintDate}                  {date}
    +{.} { \textit}                     {title}
    +{.} { }                            {part}
    +{:} { \textit}                     {subtitle}
    +{,} { \PrintEdition}               {edition}
    +{}  { \PrintEditorsB}              {editor}
    +{,} { \PrintTranslatorsC}          {translator}
    +{,} { \PrintContributions}         {contribution}
    +{,} { }                            {series}
    +{,} { \voltext}                    {volume}
    +{,} { }                            {publisher}
    +{,} { }                            {organization}
    +{,} { }                            {address}
    +{,} { }                            {status}
    +{,} { \PrintDOI}                   {doi}
    +{,} { \PrintISBNs}                 {isbn}
    +{}  { \parenthesize}               {language}
    +{}  { \PrintTranslation}           {translation}
    +{;} { \PrintReprint}               {reprint}
    +{.} { }                            {note}
    +{.} {}                             {transition}
    +{}  {\SentenceSpace \PrintReviews} {review}
}
\BibSpec{article}{%
    +{}  {\PrintAuthors}                {author}
    +{,} { \textit}                     {title}
    +{.} { }                            {part}
    +{:} { \textit}                     {subtitle}
    +{,} { \PrintContributions}         {contribution}
    +{.} { \PrintPartials}              {partial}
    +{,} { }                            {journal}
    +{}  { \textbf}                     {volume}
    +{}  { \PrintDatePV}                {date}
    +{,} { \issuetext}                  {number}
    +{,} { \eprintpages}                {pages}
    +{,} { }                            {status}
    +{,} { \PrintDOI}                   {doi}
    +{,} { \eprint}        {eprint}
    +{}  { \parenthesize}               {language}
    +{}  { \PrintTranslation}           {translation}
    +{;} { \PrintReprint}               {reprint}
    +{.} { }                            {note}
    +{.} {}                             {transition}
    +{}  {\SentenceSpace \PrintReviews} {review}
}
\BibSpec{collection.article}{%
    +{}  {\PrintAuthors}                {author}
    +{,} { \textit}                     {title}
    +{.} { }                            {part}
    +{:} { \textit}                     {subtitle}
    +{,} { \PrintContributions}         {contribution}
    +{,} { \PrintConference}            {conference}
    +{}  {\PrintBook}                   {book}
    +{,} { }                            {booktitle}
    +{,} { \PrintDateB}                 {date}
    +{,} { pp.~}                        {pages}
    +{,} { }                            {publisher}
    +{,} { }                            {organization}
    +{,} { }                            {address}
    +{,} { }                            {status}
    +{,} { \PrintDOI}                   {doi}
    +{,} { \eprint}        {eprint}
    +{}  { \parenthesize}               {language}
    +{}  { \PrintTranslation}           {translation}
    +{;} { \PrintReprint}               {reprint}
    +{.} { }                            {note}
    +{.} {}                             {transition}
    +{}  {\SentenceSpace \PrintReviews} {review}
}
\BibSpec{misc}{%
  +{}{\PrintAuthors}  {author}
  +{,}{ \textit}      {title}
  +{.}{ }             {how}
  +{}{ \parenthesize} {date}
  +{,} { available at \eprint}        {eprint}
  +{,}{ available at \url}{url}
  +{,}{ }             {note}
  +{.}{}              {transition}
}

\usepackage{amssymb, amsfonts, amsxtra, amsmath}
\usepackage{mathrsfs}
\usepackage{mathdots}
\usepackage{wasysym}
\usepackage[all]{xy}
\usepackage{bbm}
\usepackage{calc}
\usepackage{accents}

\numberwithin{equation}{section}

\newtheorem{Theorem}{Theorem}[section]
\newtheorem*{Theorem*}{Theorem}
\newtheorem{Def}[Theorem]{Definition}
\newtheorem{Lem}[Theorem]{Lemma}
\newtheorem{Prop}[Theorem]{Proposition}
\newtheorem{Cor}[Theorem]{Corollary}

\newtheorem{Rem}[Theorem]{Remark}

\newcommand\bp{\begin{proof}}
\newcommand\ep{\end{proof}}

\mathchardef\mhyph="2D

\DeclareMathOperator{\Ad}{\mathrm{Ad}}

\DeclareMathOperator{\End}{\mathrm{End}}

\DeclareMathOperator{\fin}{\mathrm{f}}

\DeclareMathOperator{\Hom}{\mathrm{Hom}}
\DeclareMathOperator{\id}{\mathrm{id}}

\DeclareMathOperator{\Rep}{\mathrm{Rep}}

\DeclareMathOperator{\Tr}{\mathrm{Tr}}

\DeclareMathOperator{\adm}{\mathrm{adm}}

\DeclareMathOperator{\Irr}{\mathrm{Irr}}

\DeclareMathOperator{\Ker}{\mathrm{Ker}}
\DeclareMathOperator{\Spec}{\mathrm{Spec}}

\DeclareMathOperator{\Lin}{\mathrm{Lin}}
\DeclareMathOperator{\Prod}{\prod}
\DeclareMathOperator{\Ann}{\mathrm{Ann}}
\DeclareMathOperator{\Mod}{\mathrm{Mod}}

\DeclareMathOperator{\coinv}{\mathrm{co}}

\DeclareMathOperator{\bounded}{\mathrm{b}}

\newcommand{\cop}{\mathrm{cop}}
\newcommand{\op}{\mathrm{op}}

\newcommand{\reg}{\mathrm{reg}}

\newcommand{\msE}{\mathscr{E}}

\newcommand{\msF}{\mathscr{F}}
\newcommand{\msG}{\mathscr{G}}

\newcommand{\msI}{\mathscr{I}}

\newcommand{\msM}{\mathscr{M}}

\newcommand{\msU}{\mathscr{U}}

\newcommand{\mfk}{\mathfrak{k}}

\newcommand{\mfsl}{\mathfrak{sl}}

\newcommand{\mfsu}{\mathfrak{su}}

\newcommand{\mcD}{\mathcal{D}}
\newcommand{\mcE}{\mathcal{E}}

\newcommand{\mcG}{\mathcal{G}}

\newcommand{\Hsp}{\mathcal{H}}
\newcommand{\Gsp}{\mathcal{G}}
\newcommand{\mcI}{\mathcal{I}}

\newcommand{\mcO}{\mathcal{O}}

\newcommand{\mcU}{\mathcal{U}}

\newcommand{\C}{\mathbb{C}}
\newcommand{\G}{\mathbb{G}}
\newcommand{\Hh}{\mathbb{H}}

\newcommand{\R}{\mathbb{R}}

\newcommand{\X}{\mathbb{X}}

\newcommand{\Z}{\mathbb{Z}}

\newcommand{\opp}{\mathrm{op}}

\title{Invariant integrals on coideals and their Drinfeld doubles}
\author{K. De Commer and J. R. Dzokou Talla}
\email{kenny.de.commer@vub.be}
\email{joel.right.dzokou.talla@vub.be}
\address{Vrije Universiteit Brussel}

\begin{document}

\begin{abstract}
Let $A$ be a CQG Hopf $*$-algebra, i.e.~ a Hopf $*$-algebra with a positive invariant state. Given a unital right coideal $*$-subalgebra $B$ of $A$, we provide conditions for the existence of a relatively invariant integral on the stabilizer coideal $B^{\perp}$ inside the dual discrete multiplier Hopf $*$-algebra of $A$. Given such a relatively invariant integral, we show how it can be extended to a relatively invariant integral on the Drinfeld double coideal. We moreover show that the representation theory of the Drinfeld double coideal has a monoidal structure. As an application, we determine the relatively invariant integral for the coideal $*$-algebra $\mcU_q(\mfsl(2,\R))$ constructed from the Podle\'{s} spheres.
\end{abstract}

%\classification{17B37; 46L65; 81R50; 16T15}

%17B37: Quantum groups, quantized enveloping algebras
%20G42: quantized function algebras
%46L65: Functional analysis, deformations, quantizations
%81R50: Quantum groups and related algebraic methods
%16T05: Hopf algebras and their applications
%16T10: Bialgebras
%16T15: Coalgebras and comodules; corings

\maketitle

\section*{Introduction}

Let $(A,\Delta)$ be a CQG (compact quantum group) Hopf $*$-algebra \cite{DK94}, to be interpreted as the algebra $A = \mcO(\G)$ of regular functions (=matrix coefficient functions) on a compact quantum group $\G$. Let $B \subseteq A$ be a unital right coideal $*$-subalgebra, so we interpret $B =\mcO(\X)$ as an algebra of functions on a compact quantum space $\X$ with transitive right action of $\G$. As we have an embedding $B \subseteq A$, we can say that $\X$ has a classical point $\bullet$ induced by the counit, so that we can view $\X = \Hh\backslash \G$ with $\Hh$ the stabilizer of $\bullet$. In general, $\Hh$ will not exist as an actual quantum subgroup, i.e.~ we do not have a Hopf $*$-algebra $\mcO(\Hh)$ with surjective Hopf $*$-algebra morphism $\mcO(\G) \twoheadrightarrow \mcO(\Hh)$. However, $\mcO(\Hh)$ still exists as a special type of coalgebra. 

Under the above assumptions, $A$ is cosemisimple, i.e.~ as a coalgebra it is simply a direct sum of finite-dimensional coalgebras dual to finite-dimensional matrix algebras. Hence dual to $A$ one has the $*$-algebra $\mcU = \widehat{A} =  \C[\G]$ consisting of functionals vanishing on all but a finite number of the above coalgebras. It can according to the context be viewed as a convolution algebra of $\G$ or as the algebra of continuous functions of compact support on the discrete quantum  group $\widehat{\G}$ which is Pontryagin-dual to $\G$. When $A$ is infinite-dimensional, $\mcU$ will not be a Hopf $*$-algebra but a \emph{multiplier Hopf $*$-algebra}, or more precisely an \emph{algebraic quantum group} \cite{VD94,VD96,VD98}. Because of its concrete nature however as a direct sum of matrix algebras, this structure can also be dealt with in a more ad hoc way. Dual to the coideal $B =\mcO(\Hh\backslash \G)$ is a left coideal $\mcI = B^{\perp} = \C[\Hh]$ inside the \emph{multiplier} algebra of $\mcU$, which we call the \emph{stabilizer coideal}. 

Recently, there has been a lot of activity in determining properties of coideals in the context of invariant integrals \cite{Tom07,FS09,KK17,Kas18,Chi18,CKS20}. For example, a necessary and sufficient condition for the existence of an invariant integral on $\mcI$ turns out to be that the coideal $B$ (or its stabilizer $\mcI$) is invariant under the antipode squared. Unfortunately, this is not always the case in all examples of interest. An important situation where this condition is not satisfied is for the \emph{Podle\'{s} spheres} at generic parameter. We will show however that the situation can be amended by considering \emph{relatively} invariant integrals (cf.~ \cite{Kas18} for similar notions). We then show that a \emph{tracial} relatively invariant integral on $\mcI$ exists when $B\subseteq A$ forms a \emph{co-Gelfand pair}, meaning that the associated unital $*$-algebra of spherical functions is commutative.  

We then consider also the Drinfeld double of a CQG coideal, which gives an analytic version of the purely algebraic construction presented in \cite{DCDz21}. This construction is important, as it gives an answer to the longstanding open problem on how to quantize real semisimple Lie groups. Indeed, it can be motivated that if $L$ is a real semisimple Lie group, realized as the real form of a simply connected, connected complex Lie group $G$, then a quantization of its convolution algebra $C_c(L)$ can be obtained as the Drinfeld double of the CQG coideal $\mcO_q(U/K)$ of $\mcO_q(U)$, where $U$ is the maximal compact subgroup of $G$ and $\mcO_q(U/K)$ is the quantum homogeneous space associated to the Letzter-Kolb coideal $*$-subalgebra \cite{Let99,Let02,Kol14,DCNTY19,DCM20} for the symmetric pair $K = L\cap U \subseteq U$. Note that this Drinfeld double coideal is \emph{a priori} only a coideal, so its category of representations would seem to lack a monoidal structure. As we demonstrate however, there \emph{is} a monoidal structure present, although its construction is not straightforward. In this paper, we do not provide the details of the above construction in general, but as a guiding example we explain the constructions in the case of quantum $SL(2,\R)$.  For further motivation, we refer to \cite{DCDz21}.

The contents of this paper are as follows. In the \emph{first section}, we recall the main preliminaries on quantum groups, and prove equivalent conditions for a coideal associated to a discrete quantum group to admit a relatively invariant integral. In the \emph{second section}, we introduce the Drinfeld double of a CQG coideal and its stabilizer dual, and show that this Drinfeld dual coideal admits as well a relatively invariant integral and associated regular representation when the original coideal does. We further show how the representation theory of the Drinfeld double coideal admits a monoidal structure. In the \emph{third section}, we apply our results to the case of quantum $SL(2,\R)$, constructed as the quantum double of a Podle\'{s} sphere and its stabilizer coideal. 
 
\emph{Acknowledgements}: The work of K. De Commer and J. R. Dzokou Talla was supported by the FWO grant G032919N. We thank A. Brochier, D. Jordan, D. Nikshych, C. Voigt and R. Yuncken for fruitful discussions. We also thank S. Neshveyev for setting us straight on using the better suited terminology `relatively invariant' instead of `quasi-invariant'. 

\section{Invariant integrals on coideals}

\subsection{CQG Hopf $*$-algebras and their duals}\label{SecBegin}

Let $(A,\Delta)$ be a CQG (compact quantum group) Hopf $*$-algebra, so $A$ is generated as an algebra by the matrix coefficients of its unitary corepresentations \cite{DK94}.  Equivalently, $A$ is a Hopf $*$-algebra with a (necessarily unique) invariant state $\Phi_A: A\rightarrow \C$, so 
\[
(\id\otimes \Phi_A)\Delta(a) = \Phi_A(a)1 = (\Phi_A\otimes \id)\Delta(a),\qquad \forall a\in A,
\]
\[
\Phi_A(1) = 1,\qquad \Phi_A(a^*a)\geq 0,\qquad \forall a\in A. 
\]
One can interpret $A = \mcO(\G)$ with $\G$ a compact quantum group. We denote by $\varepsilon$ the counit of $A$, and by $S_A$ its antipode. We further use the (unsummed) Sweedler notation $\Delta(a) = a_{(1)}\otimes a_{(2)}$. 

We refer to \cite[Chapter 1]{NT14} for all information on CQG Hopf $*$-algebras we will need. We recall the essentials in what follows. First of all, $A$ is cosemisimple as a coalgebra, and we choose a maximal family of mutually inequivalent irreducible (right) comodules $(\Hsp_{\alpha},\delta_{\alpha})$. We can assume moreover that these comodules are unitary, so the $\Hsp_{\alpha}$ are finite-dimensional Hilbert spaces and, using the usual Sweedler notation $\delta_{\alpha}(v) = v_{(0)}\otimes v_{(1)}$, we have
\[
\langle v_{(0)},w_{(0)}\rangle v_{(1)}^*w_{(1)} = \langle v,w\rangle 1,\qquad \forall v,w\in \Hsp_{\alpha}. 
\]

Let $\msU =\Lin(A,\C)$ be the vector space dual of $A$, which is a unital $*$-algebra under the convolution product and $*$-structure given by 
\[
(\omega\chi)(a) = (\omega \otimes \chi)\Delta(a),\qquad \omega^*(a) = \overline{\omega(S(a)^*)},\qquad a\in A,\omega,\chi \in \msU. 
\]
We view $\msU = \Lin(A,\C)$ as a topological space with the topology of pointwise convergence. It is sometimes convenient to use the abstract notation
\[
\tau: A \times \msU\rightarrow \C,\qquad (a,\omega)\mapsto \tau(a,\omega) = \omega(a).
\]
for the natural pairing between $\msU$ and $A$.

Any finite-dimensional unitary right $A$-comodule $(\Hsp,\delta)$ leads to a unital $*$-representation of $\msU$ on $\Hsp$ via
\[
\pi(x)v = \tau(v_{(1)},x)v_{(0)}.
\]
We call these $\msU$-representations \emph{admissible}. If conversely $\pi$ is an admissible $*$-representation of $\msU$ and $\xi,\eta\in \Hsp_{\pi}$, there exists a unique $\pi_{\xi,\eta} = \pi(\xi,\eta) \in A$ such that
\[
\tau(\pi(\xi,\eta),x) = \langle \xi,\pi(x)\eta\rangle, \qquad x \in \msU.
\]
We refer to $\pi(\xi,\eta)$ as a \emph{matrix coefficient of $\pi$}. We then also interpret
\[
\pi \in B(\Hsp)\otimes A,\qquad (\id\otimes \tau(-,h))\pi = \pi(h),\qquad h \in \msU.   
\]

We choose a  maximal family $(\Hsp_{\alpha},\pi_{\alpha})$ of mutually inequivalent irreducible admissible unital $*$-representations of $\msU$, and we write $\Irr_{\adm}(\msU) =\{\alpha\}$ for the index set. We can concretely identify $\msU$ with a direct product $*$-algebra
\begin{equation}\label{EqLinDual}
\msU\cong \Prod_{\alpha\in \Irr_{\adm}(\msU)} B(\Hsp_{\alpha}),\qquad \pi_{\alpha}: \msU \rightarrow B(\Hsp_{\alpha}),\quad \omega\mapsto (v\mapsto \pi_{\alpha}(\omega)v =  \tau(v_{(1)},\omega)v_{(0)}).
\end{equation}
The topology of $\msU$ then coincides with the product topology on $\Prod_{\alpha} \End(\Hsp_{\alpha})$. In the following, we will switch between these two pictures whenever convenient. 

We write 
\begin{equation}\label{EqAlgComplProd}
\msU \widehat{\otimes}\msU =  \Prod_{\alpha,\alpha'} \End(\Hsp_{\alpha})\otimes \End(\Hsp_{\alpha'}) \cong \Lin(A\otimes A,\C)
\end{equation} 
for the topologically completed tensor product. Then the product on $A$ dualizes to a unital $*$-homomorphism 
\[
\Delta: \msU \rightarrow \msU \widehat{\otimes}\msU. 
\]
This coproduct allows to take the tensor product of admissible representations. 

There further exists a unique $\delta_A \in \msU$ such that, for each $\alpha \in \Irr_{\adm}(\msU)$, the operator $\pi_{\alpha}(\delta_{A})\in B(\Hsp_{\alpha})$ is positive and invertible with $\Tr(\pi_{\alpha}(\delta_{A})^{1/2}) = \Tr(\pi_{\alpha}(\delta_A)^{-1/2})$ and, writing this number as $\dim_q(\alpha)$, one has the following Peter-Weyl orthogonality relations between matrix coefficients of irreducible admissible $*$-representations of $\msU$:  
\begin{equation}\label{EqPW1}
\Phi_A(\pi_{\alpha}(\xi_1,\eta_1)\pi_{\alpha'}(\xi_2,\eta_2)^*) =\delta_{\alpha,\alpha'} \frac{\langle \xi_1,\xi_2\rangle \langle \eta_2,\delta_{A}^{-1/2}\eta_1\rangle}{\dim_q(\alpha)},
\end{equation}
\begin{equation}\label{EqPW2}
\Phi_A(\pi_{\alpha'}(\xi_2,\eta_2)^*\pi_{\alpha}(\xi_1,\eta_1)) =\delta_{\alpha,\alpha'} \frac{\langle \xi_1,\delta_A^{1/2}\xi_2\rangle \langle \eta_2,\eta_1\rangle}{\dim_q(\alpha)}.
\end{equation}
In general, we can make sense of $\pi_{\alpha}(\delta_A)^z$ for any complex $z$ by positivity of $\pi_{\alpha}(\delta_A)$, and there exists then a unique $\delta_A^z \in \msU$ so that $\pi_{\alpha}(\delta_A^z) = \pi_{\alpha}(\delta_A)^z$ for all $\alpha\in \Irr_{\adm}(\msU)$ and $z\in \C$. 

It follows immediately from the orthogonality relations that $\Phi_A$ is faithful, i.e.~ both maps $A \rightarrow \msU$ given by
\begin{equation}\label{EqIsoFunct}
a\mapsto \Phi_A(a-),\qquad a\mapsto \Phi_A(-a)
\end{equation}
are injective. Moreover, they have the same range which we will write $\mcU \subseteq \msU$. Then $\mcU$ is a (non-unital) $*$-subalgebra of $\msU$, equal under the identification \eqref{EqLinDual} to $\oplus_{\alpha} \End(\Hsp_{\alpha})$. In particular, $\mcU$ is an $A$-bimodule by 
\[
\tau(b,x\lhd a) = \tau(ab,x),\qquad \tau(b,a\rhd x) = \tau(ba,x),\qquad x\in \mcU,a,b\in A.
\]

The linear map
\begin{equation}\label{EqModAntSq}
\sigma_A: A \rightarrow A,\qquad \sigma_A(\pi_{\alpha}(\xi,\eta)) = \pi_{\alpha}(\delta_A^{-1/2}\xi,\delta_A^{-1/2}\eta)
\end{equation}
satisfies $\Phi_A(ab) = \Phi_A(b\sigma_A(a))$ for all $a,b\in A$, and hence is an automorphism of $A$ called the \emph{modular automorphism}. We can then identify 
\begin{equation}\label{EqDelasFun}
\delta_A: A \rightarrow \C,\quad a\mapsto \varepsilon(\sigma_A^{-1}(a)), 
\end{equation}
so that it is clear that $\delta_A$ is a character on $A$, and in particular $S_A(\delta_A) = \delta_A^{-1}$. We further have that the antipode squared acts as follows on matrix coefficients:
\begin{equation}\label{EqAntiSqu}
S_A^2(\pi(\xi,\eta)) = \pi(\delta_A^{-1/2}\xi,\delta_A^{1/2}\eta).  
\end{equation}

Note now that the antipode $S_A$ of $A$ dualizes to an anti-isomorphism of algebras $S: \msU \rightarrow \msU$ satisfying $S(x)^* = S^{-1}(x^*)$ for $x\in \msU$. One then has 
\begin{equation}\label{EqAntiSqu2}
S^2(x) = \delta_A^{-1/2}x\delta_A^{1/2},\qquad \forall x\in \msU. 
\end{equation}
This allows one to modify $S$ as to become $*$-preserving and involutive: we define the \emph{unitary antipode} as
\[
R: \msU \rightarrow \msU,\qquad x\mapsto \delta_A^{1/4}S(x)\delta_A^{-1/4}.
\]
Then for all $x\in \msU$ one has
\begin{equation}\label{EqUnitAntip}
R(x)^* = R(x^*),\qquad R^2(x) = x,\qquad S(x) = \delta_A^{-1/4}R(x)\delta_A^{1/4},\qquad S^{-1}(x) = \delta_A^{1/4}R(x)\delta_A^{-1/4}. 
\end{equation}
This allows to give a direct formula for the \emph{dual} $\Hsp_{\overline{\pi}}$ of an admissible unitary $\msU$-representation $\Hsp_{\pi}$ as being the usual dual Hilbert space $\Hsp_{\pi}^*$ with the $*$-representation
\begin{equation}\label{EqActDual}
\overline{\pi}(x)\xi^* = (R(x)^*\xi)^*.
\end{equation}
It follows that we can write 
\begin{equation}\label{EqFormStar}
\pi(\xi,\eta)^* = \pi((\delta_A^{1/4}\xi)^*,(\delta_A^{-1/4}\eta)^*) = \pi(\delta_A^{-1/4}\xi^*,\delta_A^{1/4}\eta^*). 
\end{equation}
To end, note that $\sigma_A$ and $S_A^2$ are part of one-parameter groups of automorphisms called respectively the \emph{modular} and \emph{scaling} automorphisms groups:
\[
(\sigma_A)_z(\pi(\xi,\eta)) = \pi(\delta_A^{i\overline{z}/2}\xi,\delta_A^{-iz/2}\eta),\qquad (\tau_A)_z(\pi(\xi,\eta)) = \pi(\delta_A^{i\overline{z}/2}\xi,\delta_A^{iz/2}\eta),
\]
so that $\sigma_A^n = (\sigma_A)_{-in}$ and $S_A^{2n} = (\tau_A)_{-in}$ for $n \in \Z$. We then also write $R_A: A \rightarrow A$ for involutive the anti-$*$-(co)algebra isomorphism
\[
R_A(\pi(\xi,\eta)) = \pi(\eta^*,\xi^*),\qquad \xi,\eta\in \Hsp,\qquad S_A = R_A\circ (\tau_A)_{-i/2} = (\tau_A)_{-i/2}\circ R_A.  
\]

%We then have the natural intertwining map 
%\begin{equation}\label{EqIntDual}
%t_{\pi}: \C_{\varepsilon}\rightarrow \Hsp_{\overline{\pi}}\otimes \Hsp_{\pi},\qquad 1 \mapsto \sum_i e_i^*\otimes \delta^{1/4}e_i = \sum_i  \delta^{-1/4}e_i^*\otimes e_i,
%\end{equation}
%where $e_i$ is an orthonormal basis for $\Hsp_{\pi}$. We write $t_{\alpha} = t_{\pi_{\alpha}}$ in what follows. 

\begin{Rem}\label{RemDualCQG}
In practice, for example in the setting of quantized enveloping algebras, the Hopf $*$-algebra $A$ is obtained through the unitary pairing with another Hopf $*$-algebra $U$, 
\[
\tau: A\otimes U \rightarrow \C.
\]
This pairing is called \emph{non-degenerate} if it induces inclusions $A \subseteq \Lin(U,\C)$ and $U \subseteq \Lin(A,\C)$. One then calls a $*$-representation of $U$ on a finite-dimensional Hilbert space \emph{admissible} if it is implemented by a unitary corepresentation of $A$ through $\tau$. The non-degeneracy condition provides a dense embedding $U \subseteq \msU$, compatible with the coproduct. 
\end{Rem}

\subsection{Coideals of CQG Hopf $*$-algebras and their stabilizer duals}

Let $B \subseteq A$ be a (unital) right coideal $*$-subalgebra, so $B$ is a unital $*$-subalgebra of $A$ with 
\[
\Delta(B) \subseteq B\otimes A.
\]
As mentioned in the introduction, for $A = \mcO(\G)$ we interpret $B = \mcO(\Hh\backslash \G)$. With $B_+ = B \cap \Ker(\varepsilon)$, we can then define
\[
C = \mcO(\Hh) = A/AB_+.
\]
By \cite[Proposition 1]{Tak79}, $C$ has a unique structure of a left $A$-module coalgebra such that the projection map $\pi_C: A \mapsto C$ is a left $A$-module coalgebra map. 

Write $1_C = \pi_C(1_A)$ for the canonical grouplike element in $C$. By \cite[Theorem 3.1]{Chi18} (see also, in light of Remark \ref{RemDualCQG}, the closely related \cite[Theorem 2.2]{MS99}), we have that $C$ is cosemisimple, and we can obtain $B$ back from $C$ via
\begin{equation}\label{EqBAreCoinv}
B = \{b\in A \mid \pi_C(b_{(1)})\otimes b_{(2)} = 1_C\otimes b\}.
\end{equation}

The quotient coalgebra $C$ moreover inherits an anti-linear involution via
\[
\pi_C(a)^{\dag} := \pi_C(S(a)^*),
\]  
using e.g. that $S(AB_+)^* = AB_+$ via an observation in \cite{Ko93} (see also \cite[Lemma 1.4]{MS99}). Writing $\msI = \Lin(C,\C)$, we obtain through $\pi_C$ an inclusion of unital $*$-algebras
\[
\msI \subseteq \msU,
\]
where $\msI$ is provided with the convolution algebra structure and the $*$-structure $\omega^*(c) = \overline{\omega(c^{\dag})}$. 

\begin{Def}
We call a  $*$-representation of $\msI$ on a finite-dimensional Hilbert space \emph{admissible} if it arises as a subrepresentation of an admissible $\msU$-representation. 
\end{Def}

We can again choose a maximal set of irreducible admissible $\msI$-representations, and write them $(\mcG_{\beta},\pi_{\beta})$ for $\beta$ in some indexing set $\Irr_{\adm}(\msI)$. We have 
\begin{equation}\label{EqIDirProd}
\msI = \Prod_{\beta} \End(\Gsp_{\beta}).
\end{equation}

The coproduct of $\msU$ restricts to a unital $*$-homomorphism 
\[
\Delta: \msI \rightarrow \msU \widehat{\otimes} \msI = \prod_{\alpha,\beta} \End(\Hsp_{\alpha})\otimes \End(\Gsp_{\beta}) \cong \Lin(A\otimes C,\C),
\]
dual to the left $A$-module structure on $C$. This allows to take, within the class of admissible $*$-representations of $\msI$, the tensor product of an admissible $\msU$-representation with an admissible $\msI$-representation. 

Let now $\Phi_C \in \msI$ be the self-adjoint support projection of the admissible $*$-representation 
\[
\varepsilon_{\mid \msI}: \msI \rightarrow \C,\qquad \omega \mapsto \omega(1_C),
\]
so interpreted as an element in $\msU$ we have that $(\Phi_C)_{\alpha}\in \End(\Hsp_{\alpha})$ is the orthogonal projection onto the $\msI$-fixed vectors of $\Hsp_{\alpha}$, i.e.\ the vectors $\xi$ satisfying $h\xi = \varepsilon(h)\xi$ for all $h\in \msI$. By abuse of notation, we interpret $\Phi_C$ both as a linear map on $A$ and on $C$,
\[
\Phi_C: A \rightarrow \C,\qquad \Phi_C: C\rightarrow \C. 
\]
We have the invariance property 
\begin{equation}\label{EqInvC}
(\Phi_C\otimes \id)\Delta_C(c) = \Phi_C(c)1_C = (\id\otimes \Phi_C)\Delta_C(c),\qquad \forall c\in C. 
\end{equation}

From \eqref{EqBAreCoinv} and \eqref{EqInvC}, we obtain the surjective right $B$-linear projection map 
\begin{equation}\label{EqCondExp}
E_B: A \rightarrow B,\quad a \mapsto (\Phi_C \otimes \id)\Delta(a), 
\end{equation}
so $\Phi_C = \varepsilon \circ E_B$. In particular, $B$ is generated as a vector space by the $\pi(\xi,\eta)$ with $\Phi_C\xi = \xi$. Using \eqref{EqFormStar} and the fact that $B$ is $*$-invariant, we find that 
\begin{multline}\label{EqSPhiC}
\Phi_C S_A(\Phi_C) = S_A(\Phi_C), \qquad \Phi_C S_A^{-1}(\Phi_C) = \Phi_C,\\ 
S_A^{-1}(\Phi_C)\Phi_C = S_A^{-1}(\Phi_C),\qquad S_A(\Phi_C)\Phi_C = \Phi_C,
\end{multline}
where the second identity follows from the first by applying $S_A^{-1}$, and where the third and fourth identities follow from the first two by applying the $*$-operation. In particular, we obtain also a map
\begin{equation}\label{EqCondExp2}
F_B: A \rightarrow B,\quad a \mapsto (S_A^{-1}(\Phi_C) \otimes \id)\Delta(a).
\end{equation}
A direct calculation using again \eqref{EqFormStar} shows that 
\begin{equation}\label{EqStarE}
F_B(a) = E_B(a^*)^*,\qquad a\in A,
\end{equation}
and in particular $F_B$ is left $B$-linear. 

Let us record the following observation (see also \cite[Proposition 1.6]{LVD07}).

%, which is a direct consequence of \eqref{EqInvCom} by applying $\Phi_C$ 
\begin{Lem}\label{LemSSS}
Inside $\msU \widehat{\otimes}\msU$, the following identities hold: 
\begin{equation}\label{EqIdPhiCC}
\Delta(\Phi_C)(1\otimes \Phi_C) = \Phi_C\otimes \Phi_C= (S_A(\Phi_C)\otimes 1)\Delta(\Phi_C) = \Delta(\Phi_C)(S_A^{-1}(\Phi_C)\otimes 1).
\end{equation}
\end{Lem}
\begin{proof}
The identity $\Delta(\Phi_C)(1\otimes \Phi_C) = \Phi_C\otimes \Phi_C$ is immediate upon realizing that $\Delta(\Phi_C) \in \msU \widehat{\otimes} \msI$ by the left coideal property of $\msI$, and $\Phi_C x = x\Phi_C = \varepsilon(x)\Phi_C$ for $x\in \msI$ by construction of $\Phi_C$. 

On the other hand, for $a,b\in A$ we have 
\begin{equation}\label{EqInvCom}
\pi_C(b_{(1)})\Phi_C(ab_{(2)}) = S_A(a_{(1)})\pi_C((a_{(2)}b)_{(1)})\Phi_C((a_{(2)}b)_{(2)}) = \pi_C(S_A(a_{(1)}))\Phi_C(a_{(2)}b). 
\end{equation}
Applying $\Phi_C$, we see that 
\[
\Phi_C(b)\Phi_C(a) = \Phi_C(b_{(1)})\Phi_C(a\pi_C(b_{(2)})) = \Phi_C(b_{(1)})\Phi_C(ab_{(2)}) = \Phi_C(S_A(a_{(1)}))\Phi_C(a_{(2)}b).
\]

As this holds for all $a,b\in A$, we find $(S_A(\Phi_C)\otimes 1)\Delta(\Phi_C) = \Phi_C\otimes \Phi_C$. Applying the $*$-operation, we obtain $\Delta(\Phi_C)(S_A^{-1}(\Phi_C)\otimes 1) = \Phi_C\otimes \Phi_C$.
\end{proof}

Recall now from \cite[Proposition 2.10]{BDRV06} that $B$ allows a (unique) modular automorphism
\[
\sigma_B: B \rightarrow B,\qquad \Phi_A(bc) = \Phi_A(c\sigma_B(b)),\qquad b,c\in B. 
\]
Note that $\sigma_B$ is in general different from $(\sigma_A)_{\mid B}$. Indeed, the equality $\sigma_B = (\sigma_A)_{\mid B}$ only holds if $B$ is invariant under $\sigma_A$, which is not always the case. 

%\begin{Lem}\label{LemModAutBInv}
%The inverse $\sigma_B^{-1}$ of the modular automorphism of $B$ is given by 
%\begin{equation}\label{EqModAutBInv}
%\sigma_B^{-1}(\pi(\xi,\eta)) = \pi(\Phi_C \delta_A^{1/2}\xi,\delta_A^{1/2}\eta),\qquad \xi \in \Phi_C\Hsp_{\pi},\eta\in \Hsp_{\pi}. 
%\end{equation}
%\end{Lem} 

\begin{Lem}\label{LemModAutBInv}
The modular automorphism $\sigma_B$ of $B$ is given by $\sigma_B(b) = F_B(\sigma_A(b))$ for $b\in B$, so
\begin{equation}\label{EqModAutB}
\sigma_B(\pi(\xi,\eta)) = \pi(\delta_A^{-1/4}R(\Phi_C)\delta_A^{-1/4}\xi,\delta_A^{-1/2}\eta),\qquad \xi \in \Phi_C\Hsp_{\pi},\eta\in \Hsp_{\pi}.
\end{equation}
Its inverse $\sigma_B^{-1}$ is given by  $\sigma_B^{-1}(b) = E_B(\sigma_A^{-1}(b))$ for $b\in B$, so
\begin{equation}\label{EqModAutBInv}
\sigma_B^{-1}(\pi(\xi,\eta)) = \pi(\Phi_C \delta_A^{1/2}\xi,\delta_A^{1/2}\eta),\qquad \xi \in \Phi_C\Hsp_{\pi},\eta\in \Hsp_{\pi}. 
\end{equation}
\end{Lem} 

\begin{proof}
Since $\Phi_A$ is invariant under $\sigma_A$, since $\Phi_C(1_C) =1$ and since $E_B$ is right $B$-linear,  we have for $b,c\in B$ that 
\[
\Phi_A(bc) = \Phi_A(\sigma_A^{-1}(c)b) = \Phi_A(E_B(\sigma_A^{-1}(c)b)) = \Phi_A(E_B(\sigma_A^{-1}(c))b). 
\]
Since $(\Phi_A)_{\mid B}$ is $\sigma_B$-invariant, we deduce that $\sigma_B^{-1} = E_B \circ \sigma_A^{-1}$. By \eqref{EqModAntSq} and \eqref{EqCondExp}, we obtain \eqref{EqModAutBInv}.  

The formula for $\sigma_B$ now follows by using that $\sigma_B(b) = \sigma_B^{-1}(b^*)^*$ for $b\in B$, and invoking \eqref{EqStarE}. Formula \eqref{EqModAutB} follows again from \eqref{EqModAntSq} combined with \eqref{EqUnitAntip}.
\end{proof}
%This formula can also be deduced by following the proof of \cite[Proposition 2.10]{BDRV06}.

Note that we can make sense of 
\[
\delta_B^{1/2} := \Phi_C \delta_A^{1/2}\Phi_C
\]
as a strictly positive element in $\Phi_C\msU\Phi_C$, in the sense that this element is strictly positive in each $\End(\Phi_C \Hsp_{\pi})$. It follows that we can define $\delta_B^{z} = (\delta_B^{1/2})^{2z} \in \Phi_C\msU\Phi_C$ for $z\in \C$, and then the analytic continuation of $\sigma_B$ into a complex one-parametergroup is given by 
\[
(\sigma_B)_z(\pi(\xi,\eta)) = \pi(\delta_B^{i\overline{z}/2}\xi,\delta_A^{-iz/2}\eta),\qquad \sigma_B = (\sigma_B)_{-i}. 
\]
 In fact, using \eqref{EqUnitAntip} and \eqref{EqSPhiC} we obtain directly that 
\[
\delta_B^{-1/2}  = \delta_A^{-1/4}R(\Phi_C)\delta_A^{-1/4} \in \Phi_C \msU\Phi_C.
\]
Completing the analogy with previous notations, we then also write 
\[
(\tau_B)_z: B \rightarrow B,\qquad (\tau_{B})_z(\pi(\xi,\eta)) = \pi(\delta_B^{i\overline{z}/2}\xi,\delta_A^{iz/2}\eta),\qquad \xi\in \Phi\Hsp_{\alpha},\eta\in \Hsp_{\alpha}. 
\]
Noting that $R_A(B)$ defines a \emph{left} coideal $*$-subalgebra, we obtain anti-algebra homomorphisms
\[
S_B: B \rightarrow R_A(B),\quad b \mapsto R_A((\tau_B)_{-i/2}(b)),\qquad S_B: R_A(B) \rightarrow B,\quad R_A(b) \mapsto (\tau_B)_{-i/2}(b).
\]
Then we can write 
\begin{equation}\label{EqAntB}
S_B^2 = (\tau_B)_{-i}: B \rightarrow B.
\end{equation}

We now consider the finite support part of $\msI$. 

\begin{Def}
We put
\[
\mcI = \{y\in \msI \mid y_{\beta} =0 \textrm{ for all but finitely many }\beta\} =  \oplus_{\beta} \End(\Gsp_{\beta}) \subseteq \msI.
\]
\end{Def}

We may view $\mcI \subseteq \msU$, though in general $\mcI$ is not included in $\mcU$.

\begin{Lem}\label{LemEqPhiC}
The following equality holds:
\[
\mcI =\{\Phi_C(a-):  A \rightarrow \C\mid a\in A\} \subseteq \msI.
\]
\end{Lem} 
\begin{proof}
Consider the intertwining map
\begin{equation}\label{EqIntDual}
t_{\pi}: \C_{\varepsilon}\rightarrow \Hsp_{\overline{\pi}}\otimes \Hsp_{\pi},\qquad 1 \mapsto \sum_i e_i^*\otimes \delta_A^{1/4}e_i = \sum_i  \delta_A^{-1/4}e_i^*\otimes e_i.
\end{equation}
%where $e_i$ is an orthonormal basis for $\Hsp_{\pi}$. We write $t_{\alpha} = t_{\pi_{\alpha}}$ in what follows. 
We obtain from it the Frobenius reciprocity isomorphism
\[
\Hom_{\msI}(\pi_1,\pi\otimes \pi_2) \cong \Hom_{\msI}(\pi_2,\overline{\pi}\otimes \pi_1),\qquad T \mapsto (1\otimes T^*)(t_\pi\otimes 1),
\]
for $\pi_1,\pi_2$ admissible $\msI$-representations and $\pi$ an admissible $\msU$-representation. Applying this with $\pi_1 = \varepsilon_{\mid \msI}$, we see that for $\pi$ an admissible $\msU$-representation and $\pi'$ an admissible $\msI$-representation one has $\Hom_{\msI}(\varepsilon_{\mid \msI},\pi\otimes \pi') \neq \{0\}$ if and only if $\Hom_{\msI}(\pi',\overline{\pi}\otimes \varepsilon_{\mid\msI}) = \Hom_{\msI}(\pi',\overline{\pi}_{\mid\msI}) \neq \{0\}$. As $\overline{\pi}_{\mid\msI}$ can only contain finitely many $\pi'$ as subrepresentations, the inclusion 
\[
\mcI \supseteq \{\Phi_C(a-):  A \rightarrow \C\mid a\in A\}
\] 
follows by considering expressions $\Phi_C(ac)$ for $a,c$ matrix coefficients. 

Assume now that we would have the strict inclusion $\supsetneq$. Using invariance of $\Phi_C$, it is not hard to see that the $\Phi_C(a-)$ form a two-sided ideal in $\mcI$, so as $\mcI$ is a direct sum of matrix algebras, our assumption would imply that there exists an irreducible admissible $\msI$-representation $\pi_{\beta}$ such that $\Hom_{\msI}(\C_{\varepsilon_{\mid \msI}},\pi\otimes \pi_{\beta})=0$ for all admissible $\msU$-representations $\pi$. By Frobenius reciprocity, this says that $ \Hom_{\msI}(\pi_{\beta},\pi)=0$ for all admissible $\msU$-representations $\pi$. Clearly we obtain a contradiction by choosing $\pi$ with $\pi_{\beta} \subseteq \pi_{\mid \msI}$. 
\end{proof} 

It follows in particular that $\mcI$ is a right $A$-module by 
\[
\tau(c,x\lhd a) = \tau(ac,x),\qquad a\in A,c\in C,x\in \mcI.
\]

\begin{Rem}
In the context of Remark \ref{RemDualCQG}, assume that $I$ is a left coideal $*$-subalgebra of $U$. Then we can create from it the right coideal $*$-subalgebra 
\[
B = I^{\perp} = \{a \in A \mid \forall h\in I: \tau(a_{(1)},h)a_{(2)} = \varepsilon(h)a\}
\]
(see e.g.~ \cite[Lemma 1.2]{DCDz21}). This induces a dense embedding $I \subseteq \msI$. Again, this is a natural way in which concrete examples of coideals can be constructed.
\end{Rem}

\subsection{Relatively invariant functionals}\label{SubSecInvFunct}

We keep the notation of the previous subsection. The following notion was considered in a slightly different context in \cite{Kas18}. 

\begin{Def}
Let $g\in \msU$, so $g: A \rightarrow \C$. We say that a functional $\psi_{\mcI}:\mcI\rightarrow \C$ is \emph{$g$-invariant} if 
\[
\psi_{\mcI}(x\lhd a) = \tau(a,g)\psi_{\mcI}(x),\qquad \forall x\in \mcI,a\in A.
\]
We say that $\psi_{\mcI}$ is a (normalized) $g$-invariant integral if $\psi_{\mcI}$ is $g$-invariant and 
\[
\psi_{\mcI}(\Phi_C) = 1,\qquad \psi_{\mcI}(x^*x) \geq 0,\qquad x\in \mcI. 
\]
We say that $\psi_{\mcI}$ is \emph{relatively invariant} if it is $g$-invariant for some $g$, and \emph{invariant} if we can take $g=1$.
\end{Def}
\begin{Rem}
In \cite[Definition 1.9]{NT04} is introduced the notion of \emph{quasi-}invariant state on a C$^*$-algebra with a coaction by $\hat{A}$. Although the context is slightly different, one sees that as far as the algebraic conditions are concerned, our notion will indeed be a special case of theirs when suitably interpreted, with the Radon-Nikodym cocycle of \cite[Definition 1.9]{NT04} of the special form $1\otimes g^{-1}$. See also the upcoming Lemma \ref{LemCharg}.
\end{Rem} 

Now an invariant integral $\psi_{\mcI}$ exists when $\mcI$ is invariant under $S_A^2$, as one sees by combining \cite[Theorem 3.13]{Tom07} with \cite[Theorem 4.2]{KK17} (see also \cite[Section 4]{Chi18}). However, this condition is not satisfied in all the examples one is interested in. We show now that under an extra condition on $B$, there exists at least a non-zero \emph{relatively} invariant functional on $\mcI$. Our methods will also provide a more direct algebraic proof of the existence of invariant integrals when $\mcI$ is $S_A^2$-invariant. 

We first observe the following. 

\begin{Lem}\label{LemCharg}
Assume $\psi_{\mcI}$ is a non-zero $g$-invariant functional on $\mcI$. Then $g: A \rightarrow \C$ is a character (i.e. a unital homomorphism). Moreover, if $\psi_{\mcI}$ is an invariant integral, then $g$ is positive and invertible as an element in $\msU = \prod_{\alpha}B(\Hsp_{\alpha})$. 
\end{Lem} 
\begin{proof}
Applying the $g$-invariance condition to $\psi_{\mcI}(x\lhd (ab)) = \psi_{\mcI}((x\lhd a)\lhd b)$ for $\psi_{\mcI}(x)$ non-zero, we immediately see that $g$ is a character (clearly $g$ can not be the zero element). In particular, $g$ is invertible with inverse $S(g)$. If $\psi_{\mcI}$ is positive (and normalized), take any $a\in A$ with $\tau(a,-)$ positive as a functional on $\msU$. As $\Phi_C\in \mcI\subseteq \msU$ is positive, it then follows that also $\Phi_C \lhd a \in \mcI\subseteq \msU$ will be positive, and hence $\psi_{\mcI}(\Phi_C\lhd a) = \tau(a,g)$ must be positive. As this holds for all such $a$, this implies that $g$ is positive. 
\end{proof} 

\begin{Def}
Let $g\in \msU$ be a character on $A$. We say that $\Phi_C$ is \emph{$g$-balanced} if 
\[
\Phi_C = S(\Phi_C)g.
\]
If $g=1$, we say that $\Phi_C$ is balanced.
\end{Def}

Note that in the formalism of \cite{LVD07}, $\Phi_C$ can be seen as a group-like projection in $\msU$ (although only the case of group-like projections in $\mcU$ is considered there), and $\Phi_C$ is balanced if it is a \emph{good} group-like projection in the terminology of \cite[Definition 2.2]{FS09}. 

The following is the main theorem of this section.

\begin{Theorem}\label{TheoInvFuncIsChar}
Let $g\in \msU$ be a character. The following are equivalent: 
\begin{itemize}
\item There exists a non-zero $g$-invariant functional on $\mcI$.
\item The character $\tau(-,g)$ satisfies $\tau(b,g) = \varepsilon(\sigma_A(\sigma_B^{-1}(b))$ for all $b\in B$.
\item The projection $\Phi_C$ is $g$-balanced.
\end{itemize}
In this case, the $g$-invariant functional is unique up to multiplication by a scalar, and a canonical choice is determined by 
\[
\psi_{\mcI}(\Phi_C(a-)) = \tau(a,g). 
\] 
If moreover $g$ is positive, then $\psi_{\mcI}$ is a $g$-invariant integral. 
\end{Theorem} 

The second condition in Theorem \ref{TheoInvFuncIsChar} can be rephrased as follows.
\begin{Cor}
A relatively invariant functional exists on $\mcI$ if and only if the character $\varepsilon\circ \sigma_A\circ \sigma_B^{-1}$ on $B$ can be extended to a character on $A$. 
\end{Cor}
Using the concrete expressions for $\sigma_A$ and $\sigma_B$ and the fact that any positive grouplike element in $\msU$ commutes with the $\delta_A^{z}$, yet another way to rephrase this is as follows (for the case of relatively invariant integrals).
\begin{Cor}
A relatively invariant integral exists on $\mcI$ if and only if there exists a positive grouplike element $\delta_g^{1/2} = g\delta_A^{1/2}$ such that 
\[
\Phi_C \delta_g^{1/2} = \delta_g^{1/2}\Phi_C,\qquad \Phi_C\delta_A^{1/2}\Phi_C = \Phi_C\delta_g^{1/2}\Phi_C. 
\]
\end{Cor}

We currently do not have any examples of coideals for which there exists no relatively invariant integral. 

Let us now come to the proof of Theorem \ref{TheoInvFuncIsChar}. We start with some preparations. We follow the usual conventions for Sweedler notation: if $(M,\delta)$ is a left $A$-comodule, we write $\delta(m) = m_{(-1)}\otimes m_{(0)}$, and similarly if $(M,\delta)$ is a right comodule we write $\delta(m) = m_{(0)}\otimes m_{(1)}$. 

Recall from \cite[Theorem 2]{Tak79} that, with ${}_B\Mod$ the category of left $B$-modules and ${}_A^C\Mod$ the category of relative left $(C,A)$-comodules, the functors 
\[
A\underset{B}{\otimes}-: {}_B\Mod \rightarrow {}_A^C\Mod, \qquad \coinv_C: {}_A^C\Mod \rightarrow {}_B\Mod,\quad  M\mapsto M^{\coinv} = \{m \in M \mid m_{(-1)}\otimes m_{(0)}= 1\otimes m\}
\]
are quasi-inverses of each other. We will need a similar result for right modules/comodules. We write $\op$ for opposite algebras and $\cop$ for opposite coalgebras. Note first that also $C^{\cop}$ is cosemisimple, and that we have a coalgebra projection map 
\[
\pi_{C^{\cop}}: A \rightarrow C^{\cop},\quad a \mapsto \pi_C(S_A(a)),
\]
making $C^{\cop}$ into a left $A^{\op}$-module coalgebra via $a^{\op}\cdot c = S_A(a)c$. In particular, $A$ is a left $C^{\cop}$-comodule by 
\[
a \mapsto \pi_C(S_A(a_{(1)})) \otimes a_{(2)}.
\]

\begin{Lem}
The space of coinvariants for the left $C^{\cop}$-comodule structure on $A$ is $S_A^{-2}(B)$.
\end{Lem} 
\begin{proof}
An element $a\in A$ is left $C^{\cop}$-coinvariant if and only if 
\[
\pi_C(S_A(a_{(1)})) \otimes a_{(2)} = \pi_C(1) \otimes a.
\]
 Now as $\Ker(\pi_C)$ is $S_A(-)^*$-invariant, we see that this will hold if and only if 
\[
\pi_C(S_A^2(a_{(1)})^*) \otimes S_A^2(a_{(2)})^* = \pi_C(1)\otimes S_A^2(a)^*,
\] 
i.e.~ $S_A^2(a)^* \in B$ and $a \in S_A^{-2}(B)$. 
\end{proof}

Let us call a right $A$-module and right $C$-comodule $(M,\delta)$ a \emph{relative} right $(A,C)$-comodule if 
\[
(1\otimes a_{(1)})\delta(ma_{(2)}) = \delta(m)(a\otimes 1),\qquad \forall a\in A,m\in M,
\]
which is equivalent to $\delta(ma) = m_{(0)}a_{(2)}\otimes S_A(a_{(1)})m_{(1)}$. We write $\Mod_A^C$ for the resulting category. Clearly we have an equivalence between the relative right $(A,C)$-comodules and the (usual) relative left $(A^{\op},C^{\cop})$-comodules by $(M,\delta) \leftrightarrow (M,\delta^{\cop})$. It then follows from the above and \cite[Theorem 2]{Tak79} that we obtain an equivalence of categories
\begin{equation}\label{EqEquivMod}
- \underset{B}{\otimes}\widetilde{A}: \Mod_B \cong \Mod_A^C,\qquad \coinv_C: \Mod_A^C \cong \Mod_B, 
\end{equation}
where $\widetilde{A}$ is $A$ with the left $B$-module structure $b \cdot a = S_A^{-2}(b)a$ and the relative right $(A,C)$-comodule structure given by right multiplication with $A$ and $\delta(a) = a_{(2)}\otimes \pi_C(S(a_{(1)}))$. Note further that if $M\in \Mod_A^C$, we view $M^{\coinv_C}$, the space of $C$-coinvariant elements, as a right $B$-module via $m\cdot b = mS_A^{-2}(b)$.

\begin{Lem}\label{LemEqualFunct}
The map 
\[
\theta: B \rightarrow A,\qquad b\mapsto \sigma_B^{-1}\sigma_AS_A^{-2}(b)
\]
restricts to an automorphism of $B$, and for all $b\in B$ one has 
\[
\Phi_C(S_A^{-2}(b)) = \varepsilon(\theta(b)).
\]
\end{Lem}
\begin{proof}
It is clear that $\theta$ maps $B$ into itself by the concrete formulas in \eqref{EqModAntSq}, \eqref{EqAntiSqu} and \eqref{EqModAutB}, with also 
\begin{equation}\label{EqFormTheta}
\theta(b) = \sigma_AS_A^{-2}\sigma_B^{-1}(b) = S_A^{-2}\sigma_A\sigma_B^{-1}(b),\qquad b\in B. 
\end{equation}
Take then $\xi,\eta\in \Hsp_{\pi}$ and put $b = \pi(\Phi_C\xi,\eta)$. Such $b$ are general elements in $B$, and it is sufficient to prove for these that $\Phi_C(\sigma_A^{-1}\sigma_B(b)) = \varepsilon(b)$. But using again \eqref{EqModAntSq} and \eqref{EqModAutB} one finds 
\[
\Phi_C(\sigma_A^{-1}\sigma_B(\pi(\Phi_C\xi,\eta))) = \langle \delta_A^{1/4}R(\Phi_C)\delta_A^{-1/4}\Phi_C\xi,\Phi_C\eta\rangle = \langle S^{-1}(\Phi_C)\Phi_C\xi,\Phi_C\eta\rangle, 
\]
which by $\Phi_C^* = \Phi_C$ and  \eqref{EqSPhiC} reduces to $ \langle \Phi_C\xi,\Phi_C\eta\rangle = \varepsilon(\pi(\Phi_C\xi,\eta)) = \varepsilon(b)$. 
\end{proof}

\begin{Prop}\label{PropCharAnn}
Consider the right $A$-module $\Ann_r(\Phi_C) = \{a\in A\mid \Phi_C(a-)=0\}$. Then
\[
\Ann_r(\Phi_C)  = \sigma_A^{-1}(\sigma_B(B_+))A.
\]
\end{Prop}
\begin{proof}
It follows from Lemma \ref{LemEqPhiC} and \eqref{EqInvCom} that the right $A$-module $\mcI$ can be made into a relative right $(C,A)$-comodule by 
\[
\Phi_C(a-) \mapsto \Phi_C(a_{(2)}-)\otimes \pi_C(S_A(a_{(1)})),
\]
and that $\Ann_r(\Phi_C)$ is a relative right $(C,A)$-subcomodule of $\widetilde{A}$, leading to a short exact sequence
\[
\Ann_r(\Phi_C) \rightarrow \widetilde{A}\rightarrow \mcI,
\]
where the last map is given by $a\mapsto \Phi_C(a-)$. 

Now by applying $\Phi_C$ to \eqref{EqInvCom}, we see that $\Phi_C(a-)$ will be coinvariant if and only if $\Phi_C(a-) =\Phi_C(a)\Phi_C$. In particular, it follows that $\mcI^{\coinv_C}$ is a one-dimensional right $B$-module spanned by $\Phi_C$, with 
\[
\Phi_C \cdot b  = \Phi_C(S_A^{-2}(b)-) = \Phi_C(S_A^{-2}(b))\Phi_C. 
\]
By \eqref{EqEquivMod}, we deduce that $\Ann_r(\Phi_C) = S_A^{-2}(\Ker((\Phi_C\circ S_A^{-2})_{\mid B}))A$. But by Lemma \ref{LemEqualFunct} we have 
\[
\Ker((\Phi\circ S_A^{-2})_{\mid B}) = S_A^2(\sigma_A^{-1}(\sigma_B(B_+))),
\]
which finishes the proof. 
\end{proof} 

We are now ready to prove Theorem \ref{TheoInvFuncIsChar}. We will use in the course of the proof again the map 
\[
\Delta: \msU \rightarrow \msU \widehat{\otimes}\msU := \Lin(A\otimes A,\C).
\]
Then characters on $A$ correspond to non-zero grouplike elements in $\msU$ with respect to $\Delta$.

\begin{proof}[Proof (of Theorem \ref{TheoInvFuncIsChar})]
Assume given a general non-zero $g$-invariant functional $\psi_{\mcI}$ on $\mcI$. The $g$-invariance of $\psi_{\mcI}$ gives $\psi_{\mcI}(\Phi_C(a-)) = \tau(a,g)\psi_{\mcI}(\Phi_C)$, so the space of $g$-invariant functionals is one-dimensional. Upon rescaling $\psi_{\mcI}$ if necessary, we may assume $
\psi_{\mcI}(\Phi_C(a-)) = \tau(a,g)$.  By Lemma \ref{LemEqPhiC} and Proposition \ref{PropCharAnn}, such a functional is well-defined on $\mcI$ if and only if 
\begin{equation}\label{EqIdgHom}
\tau(\sigma_A^{-1}\sigma_B(b),g) = \varepsilon(b),\qquad \forall b\in B,
\end{equation}
which by \eqref{EqFormTheta} is the same as the second condition in Theorem \ref{TheoInvFuncIsChar}. Conversely, if the second condition in Theorem \ref{TheoInvFuncIsChar} is satisfied, we can by the above define $\psi_{\mcI}(\Phi_C(a-)) = \tau(a,g)$, and as $g$ is assumed to be a character it is then clear that $\psi_{\mcI}$ is a non-zero $g$-invariant functional. 

Now a general element of $B$ can be written $\pi(\Phi_C\xi,\eta)$ with $\xi,\eta \in \Hsp_{\pi}$ arbitrary. From \eqref{EqModAutB} we see that concretely
\begin{equation}\label{EqModMod}
\sigma_A^{-1}\sigma_B(\pi(\Phi_C\xi,\eta)) = \pi(\delta_A^{1/4}R(\Phi_C) \delta_A^{-1/4}\Phi_C\xi,\eta) =  \pi(S^{-1}(\Phi_C)\Phi_C\xi,\eta) =\pi(S^{-1}(\Phi_C)\xi,\eta),
\end{equation}
where in the last equation we used Corollary \ref{LemSSS}. We can hence rewrite \eqref{EqIdgHom} as 
\[
\langle S^{-1}(\Phi_C)\xi,g\eta\rangle = \langle \Phi_C\xi,\eta\rangle,\qquad \forall \xi,\eta\in \Hsp_{\pi}.
\]
This can clearly hold in general if and only if $S(\Phi_C)g = \Phi_C$, i.e. $\Phi_C$ is $g$-balanced.

Assume now that $g$ is positive. As $g$ is a character, we have $S^2(g) = g$, so $g$ and any power of $\delta_A$ commute. Moreover, from the $g$-balancedness of $\Phi_C$ we deduce that 
\[
S^2(\Phi_C) = S(\Phi_Cg^{-1}) = g\Phi_Cg^{-1}. 
\]If we define 
\[
\kappa_z = \Ad_{(g\delta_A^{1/2})^{-iz}},\qquad z\in \C,
\]
then $\kappa_z$ is an algebra automorphism of $\mcU$ (and of $\msU$) such that
\begin{equation}\label{EqkappaPhi}
\kappa_z(\Phi_C) = \Phi_C,\qquad \Delta(\kappa_z(X))=  (\kappa_z\widehat{\otimes}\kappa_z)\Delta(X),\qquad X\in \msU.  
\end{equation}
Since $\mcI = \{\Phi_C(a-)\mid a\in A\}$ by Lemma \ref{LemEqPhiC}, we see that $\kappa_z$ restricts to a complex one-parametergroup of algebra automorphisms of $\mcI$. It follows that there exists a positive invertible element $\mu\in \msI$ such that 
\[
\kappa_z(x) = \mu^{iz} x\mu^{-iz},\qquad x\in \msI.
\]
In the following, we write $\kappa = \kappa_{-i}$. 

We claim that upon possibly rescaling each $\mu_{\beta}$ with a positive scalar, we have
\begin{equation}\label{EqExprDelPhiC}
\Delta(\Phi_C) = \sum_{\beta,i,j}  S(\mu^{-1} e_{ji}^{\beta})g\otimes e_{ij}^{\beta},
\end{equation}
where $e_i^{\beta}$ is an orthonormal basis of $\Gsp_{\beta}$ and $e_{ij}^{\beta}$ the associated matrix units, the sum converging pointwise on $A\otimes C$. Indeed, we can write 
\[
\Delta(\Phi_C) = \sum_{\beta,i,j} z_{ij}^{\beta}\otimes e_{ij}^{\beta}
\]
with $z_{ij}^{\beta} \in \msU$. By $g$-balancedness, we however also have that 
\[
\Delta(\Phi_C) = (S\otimes S)(\Delta^{\opp}(\Phi_C))(g\otimes g) \in S(\msI)g\widehat{\otimes} \msU,
\]
so necessarily $z_{ij}^{\beta} = S(t_{ij}^{\beta})g$ for $t_{ij}^{\beta} \in \msI$. Since 
\[
\Delta(\Phi_C)(1\otimes y) = \Delta(\Phi_C)(S^{-1}(y)\otimes 1),\qquad (1\otimes y)\Delta(\Phi_C) = (S(y)\otimes 1)\Delta(\Phi_C),\qquad y\in \msI,
\]
 we see that for all indices
\begin{equation}\label{EqIdDett}
 \kappa^{-1}(e_{rs}^{\alpha}) t_{ij}^{\beta} = \delta_{\alpha,\beta}\delta_{j,s}t_{ir}^{\alpha},\qquad t_{ij}^{\beta}e_{rs}^{\alpha} = \delta_{\alpha,\beta} \delta_{i,r} t_{sj}^{\alpha}.
\end{equation}
Now since the second leg of $\Delta(\Phi_C)$ spans $\mcI$, not all $t_{ij}^{\beta}$ are zero for a fixed $\beta$. It then follows straightforwardly that $t_{ij}^{\beta}$ must be a constant non-zero multiple $c_{\beta}$ of $\mu^{-1} e_{ji}^{\beta}$. As $\Delta(\Phi_C)$ needs to be a projection, we easily see that for each $\beta$
\begin{equation}\label{EqSumNorm}
c_{\beta} \sum_{i,j} e_{ji}^{\beta}g^{-1}\mu^{-1}e_{ij}^{\beta} = 1_{\beta},
\end{equation}
so necessarily $c_{\beta}>0$. Rescaling $\mu_{\beta}$ with $c_{\beta}$, we obtain \eqref{EqExprDelPhiC}. But applying the $g$-invariance of $\psi_{\mcI}$ to this expression, we find
\[
\psi_{\mcI}(x) = \sum_{\beta} \Tr_{\beta}(\mu_{\beta}x),\qquad x\in \mcI,
\]
hence $\psi_{\mcI}$ is positive. 
\end{proof}

Let us note the following consequence of the proof of Theorem \ref{TheoInvFuncIsChar}. 

\begin{Cor}
A $\delta_A^{-1/2}$-invariant integral $\psi_{\mcI}$ exists on $\mcI$ if and only if $\Phi_C\delta_A^{1/2}\Phi_C = \Phi_C$, and in this case $\psi_{\mcI}$ is tracial. 
\end{Cor}

A particular case to which the above corollary applies is the following. 

\begin{Def}
We say that $\mcI \subseteq \msU$ is a \emph{Gelfand pair} if $\Phi_C\msU\Phi_C$ is commutative. We say that $B \subseteq A$ is a  \emph{co-Gelfand pair} if the unital $*$-algebra of $*$-spherical functions
\[
\Xi = \{\pi(\Phi_C \xi,R(\Phi_C)\eta)\mid \pi \textrm{ admissible},\xi,\eta\in \Hsp_{\pi}\}.
\]
is commutative.
\end{Def}
Note that by the formula for the $*$-structure on $A$ together with \eqref{EqSPhiC}, we have indeed that $\Xi$ is a $*$-algebra, which can also be written as 
\[
\Xi = \{\pi(\Phi_C \xi,\delta_A^{1/4}\Phi_C\eta)\mid \pi \textrm{ admissible},\xi,\eta\in \Hsp_{\pi}\}.
\]
Hence we have shifted the algebra of spherical functions as defined in (say) \cite[Section 4]{Let03} with the character $\delta_{A}^{1/4}$ to make it $*$-invariant. This coincides with the definition in other sources such as \cite[Chapter 11]{KS97}, where one looks for functions invariant with respect to a (2-sided) coideal in $\msU$ interacting with the $*$-structure in a different way then by $*$-invariance (and the shift with $\delta_A^{1/4}$ is implemented there then already on the $\msU$-side). 

\begin{Cor}
If $B \subseteq A$ is a co-Gelfand pair, then $\mcI$ admits a $\delta_A^{-1/2}$-invariant trace. 
\end{Cor}
\begin{proof}
We are to show that $\delta_B^{1/2} = \Phi_C$. However, it follows immediately from the orthogonality relations \eqref{EqPW1} and \eqref{EqPW2} that $(\Phi_{A})_{\mid \Xi}$ admits (in general) the modular automorphism 
\[
\sigma_{\Xi}(\pi(\xi,\eta)) = \pi(\delta_B^{-1/2}\xi,R(\delta_B^{1/2})\eta). 
\]
Since by assumption $\Xi$ is commutative, it follows that $\sigma_{\Xi} = \id$ and hence $\delta_B^{-1/2} = \Phi_C = \delta_B^{1/2}$.
\end{proof}

\section{Invariant functionals on the Drinfeld double coideal}\label{SecInfFuncDD}

\subsection{Drinfeld double coideals}

Let $A$ be a CQG Hopf $*$-algebra with dual $\mcU \subseteq \msU$. We can construct a $*$-algebra $\mcD(A,\msU)$ generated by $A$ and $\msU$ with commutation relations given by 
\[
xa = a_{(2)}(S_A^{-1}(a_{(1)})\rhd x\lhd a_{(3)}),\qquad x\in \msU,a\in A. 
\]
Inside, we have the \emph{Drinfeld double $\mcD(A,\mcU)$ of $A$}, spanned by elements of the form $ax$ for $a\in A$ and $x\in \mcU$. The multiplication maps
\[
A\otimes \mcU \rightarrow \mcD(A,\mcU),\qquad \mcU \otimes A \rightarrow \mcD(A,\mcU)
\]
are then linear isomorphisms. One can make $\mcD(A,\mcU)$ into a \emph{multiplier Hopf $*$-algebra with (positive) invariant integrals} \cite{VD98,DrVD01,DeVD04} containing, suitably interpreted, $(A,\Delta)$ and $(\mcU,\Delta^{\opp})$ as copies. As we will not explicitly need this extra structure, we refrain from spelling out the details. 

Assume now that $B\subseteq A$ is a coideal $*$-subalgebra $B$, and let $\mcI \subseteq \msI \subseteq \msU$ be the stabilizer dual. Then $\mcD(A,\msU)$ contains a subalgebra $\mcD(B,\msI)$ generated by $B$ and $\msI$. The `orthogonality' between $B$ and $\msI$  gives that the commutation relations simplify to 
\[
yb = b_{(1)}(y\lhd b_{(2)}), \qquad by = (y\lhd S_A^{-1}(b_{(2)}))b_{(1)},\qquad b\in B, y\in \msI. 
\]
The $*$-subalgebra $\mcD(B,\mcI)$ spanned by elements of the form $yb$ with $y\in \mcI$ and $b\in B$ is called the  \emph{quantum double coideal} of $B$, and one can easily check that also in this case the multiplication maps
\[
B\otimes \mcI \rightarrow \mcD(B,\mcI),\qquad \mcI \otimes B \rightarrow \mcD(B,\mcI)
\]
are linear isomorphisms. The $*$-algebra $\mcD(B,\mcI)$ can be considered as a right coideal $*$-subalgebra in (the multiplier $*$-algebra of) $\mcD(A,\mcU)$, but again we will not have any need to spell this out explicitly. Note however that $\mcD(B,\mcI)$ is naturally an $\msU$-$A$-bimodule via 
\[
x \rhd (yb) = y(x\rhd b),\qquad (yb) \lhd a = (y\lhd a)b,\qquad x\in \msU,a \in A, y \in \mcI,b\in B,  
\]
where $x\rhd b = (\id\otimes \tau(-,x))\Delta(b)$. 

We observe the following easy lemma. 

\begin{Lem}\label{LemRepIsBound}
Any $*$-representation of $\mcD(B,\mcI)$ on a pre-Hilbert space is bounded.
\end{Lem} 
\begin{proof}
If $(V,\pi)$ is a $*$-representation of $\mcD(B,\mcI)$, then it is in particular an $\mcI$-representation. Since $\mcI$ is a direct sum of matrix algebras, it follows that this is necessarily a bounded $*$-representation. From the  commutation relations it follows immediately that $\mcI \Hsp$ carries a $*$-representation $\pi$ of $B$ such that $\pi(b)\pi(x) =\pi(bx)$ for all $b\in B,x\in \mcI$. It is now  sufficient to observe that any $*$-representation of $B$ on a pre-Hilbert space is bounded, which follows from the fact that, for $e_i$ an orthonormal basis of $\Hsp_{\alpha}$,
\[
\sum_i \pi_{\alpha}(\xi,\delta_A^{1/4} e_i)^*\pi_{\alpha}(\xi,\delta_A^{1/4}e_i) = \|\delta_A^{1/4}\xi\|^21_A,\qquad \xi\in \Phi_C\Hsp_{\alpha},
\] 
cf.\ \cite{Boc95}. 
\end{proof} 

We further keep all notation as in the previous section. Assume that $g\in \msU$ is a positive grouplike element, and that $\Phi_C$ is $g$-balanced. Let $\psi_{\mcI} = \sum_{\beta} \Tr(\mu_{\beta}-)$ be the associated $g$-invariant positive functional on $\mcI$, with associated one-parametergroup $\kappa_z(x) = \mu^z x \mu^{-z}$ for $x\in \msI$. We also write again $\kappa = \kappa_{-i}$.  

\begin{Prop}\label{PropInvFunct}
The functional 
\[
\varphi_{\mcD}: \mcD(B,\mcI) \rightarrow \C,\qquad xb \mapsto \psi_{\mcI}(x)\Phi_A(b),\qquad b\in B,x\in \mcI
\]
is a positive functional on $\mcD(B,\mcI)$. Moreover, $\varphi_{\mcD}$ is modular with modular automorphism given by 
\begin{equation}\label{EqModAut}
\sigma_{\mcD}(yb) = \kappa(y)\sigma_B(g^{-1}\rhd b), \qquad y\in \mcI,b\in B.
\end{equation}

The functional $\varphi_{\mcD}$ is $g$-invariant in the sense that 
\[
\varphi_{\mcD}(x \rhd (yb) \lhd a) =\varepsilon(x) \varphi_{\mcD}(yb) \tau(a,g).  
\]
\end{Prop} 
Modularity of $\varphi_{\mcD}$ means that 
\[
\varphi_{\mcD}(xy) =\varphi_{\mcD}(y\sigma_{\mcD}(x)),\qquad \forall x,y\in \mcD(B,\mcI). 
\]
\begin{proof}
We easily compute for $x,y\in \mcI$ and $b,c \in B$ that
\begin{equation}\label{EqvarphiProd}
\varphi_{\mcD}((cy)^*bx) = \varphi_{\mcD}(y^*(x \lhd (S_A^{-1}(c^*b)_{(2)}))(c^*b)_{(1)}) = \psi_{\mcI}(y^*x)\Phi_A(c^*b),
\end{equation}
from which positivity of $\varphi_{\mcD}$ follows. 

The modularity of $\varphi_{\mcD}$ with respect to the proposed modular automorphism follows from the following computations for $x,y,z\in \mcI$ and $b,c,d\in B$: 
\[
\varphi_{\mcD}(yxb) = \psi_{\mcI}(yx)\Phi_A(b) = \psi_{\mcI}(x\kappa(y))\Phi_A(b),
\]
\[
\varphi_{\mcD}(xbz) = \varphi_{\mcD}(x(z\lhd S_A^{-1}(b_{(2)}))b_{(1)}) =  \psi_{\mcI}(x (z\lhd S_A^{-1}(b_{(2)})))\Phi_A(b_{(1)}) = \psi_{\mcI}(xz)\Phi_A(b),
\]
\[
\varphi_{\mcD}(cxb) = \psi_{\mcI}(x\lhd S_A^{-1}(c_{(2)}))\Phi_A(c_{(1)}b) = \tau(S_A^{-1}(c_{(2)}),g)\psi_{\mcI}(x)\Phi_A(c_{(1)}b) = \psi_{\mcI}(x) \Phi_A(b \sigma_B(g^{-1}\rhd c)),
\]
\[
\varphi_{\mcD}(xbd) = \psi_{\mcI}(x)\Phi_A(bd). 
\]
The $g$-invariance of $\varphi_{\mcD}$ is immediate. 
\end{proof}

\begin{Cor}\label{CorDrinfDoubTrace}
If $\Phi_C$ is $\delta_A^{-1/2}$-balanced, then $\varphi_{\mcD}$ is a positive $\delta_A^{-1/2}$-invariant trace. 
\end{Cor}
\begin{proof}
We only need to check that $\varphi_{\mcD}$ is tracial. But since $\psi_{\mcI}$ is tracial in this case, $\kappa = \id$. Moreover, $\sigma_B(\delta_A^{1/2}\rhd b)) = b$ for $b\in B$ since $\delta_B = 1$. The result follows. 
\end{proof}

Let now $L^2(B)$ be the GNS-space for $B$ with respect to $\Phi_A$, so $L^2(B)$ is the completion of $B$ with respect to the inner product $\langle b,c\rangle =\Phi_A(b^*c)$. Write $\Lambda_B:B\rightarrow L^2(B)$ for the GNS map. Similarly, let $L^2(\mcI)$ be the GNS-space for $\mcI$ with respect to $\psi_{\mcI}$, with GNS-map $\Lambda_{\mcI}$. 

\begin{Theorem}
There exists a unique non-degenerate $*$-representation
\[
\pi_{\reg}: \mcD(B,\mcI) \rightarrow B(L^2(B)\otimes L^2(\mcI))
\]
such that for $b,c\in B$ and $x,y\in \mcI$
\begin{equation}\label{EqFormRegRep}
\begin{split}
\pi_{\reg}(c)(\Lambda_B(b)\otimes \Lambda_{\mcI}(x)) &=\Lambda_B(cb)\otimes \Lambda_{\mcI}(x)  ,\\\pi_{\reg}(y)(\Lambda_B(b)\otimes \Lambda_{\mcI}(x)) &=   \Lambda_B(b_{(1)})\otimes \Lambda_{\mcI}((y\lhd b_{(2)})x). 
\end{split}
\end{equation}
\end{Theorem}

\begin{proof}
Consider the map 
\[
\Lambda_{\mcD}: \mcD(B,\mcI) \rightarrow L^2(B)\otimes L^2(\mcI),\qquad bx \mapsto \Lambda_B(b)\otimes \Lambda_{\mcI}(x). 
\]
Then \eqref{EqvarphiProd} shows $\langle \Lambda_{\mcD}(x),\Lambda_{\mcD}(y)\rangle = \varphi_{\mcD}(x^*y)$. It follows immediately that we can define a non-degenerate $*$-representation of $\mcD(B,\mcI)$  on $\Lambda_B(B)\otimes \Lambda_{\mcI}(\mcI)$ such that $x\Lambda_{\mcD}(y) = \Lambda_{\mcD}(xy)$. This extends to a bounded $*$-representation on $L^2(B)\otimes L^2(\mcI)$ by Lemma \ref{LemRepIsBound}. An easy computation shows that the formulas \eqref{EqFormRegRep} hold. 
\end{proof}

\begin{Def}\label{DefRegRep}
We call $\pi_{\reg}$ the \emph{regular $*$-representation} of $\mcD(B,\mcI)$ (with respect to $g$).
\end{Def}

\subsection{Monoidal structure for $\Rep(\mcD(B,I))$ and categorical interpretation}

Let $A$ be a CQG Hopf $*$-algebra with dual $\mcU \subseteq \msU$. Let $B \subseteq A$ be a right coideal $*$-subalgebra with stabilizer dual $\mcI \subseteq \msI \subseteq \msU$. The goal of this section is to look for a natural tensor product for $*$-representations of $\mcD(B,\mcI)$. We thank A. Brochier, D. Jordan and D. Nikshych for discussions on this subject.

We motivate first the existence of a monoidal structure from purely algebraic considerations. Assume in fact that $A$ is a general bialgebra, and $B \subseteq A$ a right coideal subalgebra. Write $C = A/AB_+$ for the associated left $A$-module coalgebra with coproduct $\Delta_C$ and projection map 
\[
\pi_C: A \rightarrow C,\qquad a \mapsto [a],
\] 
where we use the bracket-notation also for more general quotient maps in what follows.  

We can then consider the categories 
\begin{equation}\label{EqCats}
{}_B\msM^C,\qquad {}_A^C\msM^C,\qquad {}_B\msM_B^A
\end{equation}
of (bi)modules relative to the comodule structure(s). For example, ${}_B\msM^C$ consists of \emph{Doi-Koppinen modules} \cite{Doi92,Kop95}, which are left $B$-modules $M$ with a right $C$-comodule $\delta_M$ structure such that 
\[
\delta_M(bm) = \Delta(b)\delta_M(m),\qquad b\in B,m\in M.
\]
One calls ${}_A^C \msM^C$ the category of $A$-relative $C$-bicomodules, and ${}_B\msM_B^A$ the category of $A$-relative $B$-bimodules. 
%Note that this compatibility relation is meaningful as $B$ is a right $A$-comodule algebra and $C$ is a left $A$-module coalgebra. 

There is a natural adjoint functor pair
\begin{equation}\label{EqAdjPair1}
F_B: {}_B\msM^C \rightarrow {}_A^C\msM^C,\quad V \mapsto {}^{\bullet}_{\bullet}A^{\bullet}\underset{B}{\otimes} V^{\bullet}\quad \dashv \quad  G_B:{}_A^C\msM^C \rightarrow {}_B\msM^C ,\quad M \mapsto {}^{\coinv_C}\!\!\!{}_{\bullet}M^{\bullet},  
\end{equation}
where the bullets indicate the natural places where the (co)algebras (co)act. The unit/counit adjunction maps are given by
\begin{equation}\label{EqAdjPairUn}
\varepsilon_M^{B}:  A\underset{B}{\otimes} \left({}^{\coinv_C}M\right)  \rightarrow M,\quad a\otimes m \mapsto am,\qquad 
\eta_V^{B}: V \rightarrow {}^{\coinv_C}(A\underset{B}{\otimes} V),\quad v \mapsto 1\otimes v.
\end{equation}
Similarly, with  $\overset{C}{\square}$ denoting the cotensor product, there is an adjoint functor pair
\begin{equation}\label{EqAdjPair2}
F_C: {}_B\msM_B^A\rightarrow {}_B\msM^C,\quad M \mapsto M/MB_+\quad \dashv\quad 
G_C: {}_B\msM^C \rightarrow {}_B\msM_B^A,\quad V \mapsto {}_{\bullet}V \overset{C}{\square} {}_{\bullet}A_{\bullet}^{\bullet},
\end{equation}
with unit/counit adjunction maps
\begin{multline}\label{EqAdjPairUn2}
\varepsilon_V^C: (V\overset{C}{\square} A)/(V\overset{C}{\square} A)B_+ \rightarrow V\quad [\sum_i v_i \otimes a_i] \mapsto \sum_i \varepsilon(a_i)v_i,\\ \eta_M^C: M \rightarrow (M/MB_+)\overset{C}{\square} A,\quad m \mapsto [m_{(0)}]\otimes m_{(1)}. 
\end{multline}

Moreover, ${}_A^C\msM^C$ and ${}_B\msM_B^A$ are monoidal through respectively the cotensor product $\overset{C}{\square}$ with the diagonal module structure and $\underset{B}{\otimes}$ with the diagonal comodule structure, and the composite adjunction functors $F_B\circ F_C$ and $G_C\circ G_B$ are resp.\ oplax/lax monoidal through the maps
\begin{multline}\label{EqOpLax}
A \underset{B}{\otimes} \left(M\underset{B}{\otimes} N / (M\underset{B}{\otimes} N)B_+\right) \rightarrow (A\underset{B}{\otimes} M/MB_+)\overset{C}{\square} (A\underset{B}{\otimes} N/NB_+),\\ a \otimes [m\otimes n] \rightarrow (a_{(1)} \otimes [m_{(0)}]) \otimes (a_{(2)}m_{(1)} \otimes [n]),
\end{multline} 
\begin{multline}\label{EqLax}
(({}^{\coinv_C}M)\overset{C}{\square} A) \underset{B}{\otimes} ({}^{\coinv_C}N)\overset{C}{\square} A)  \rightarrow {}^{\coinv_C}(M \overset{C}{\square} N) \overset{C}{\square} A,\\ (\sum_i m_i\otimes a_i) \otimes (\sum_j n_j\otimes a_j') \mapsto\sum_j (\sum_{i} m_i\otimes a_{i(1)}n_j) \otimes a_{i(2)}a_j'. 
\end{multline}

Assume now that $A$ is a Hopf algebra, with ${}_BA$ faithfully flat and $A^C$ faithfully coflat. Then it follows from \cite[Theorem 1 and Theorem 2]{Tak79} that the above adjoint pairs are equivalences,
\begin{equation}\label{EqEquiv} 
{}_B\msM_B^A \cong {}_B\msM^C\cong {}_A^C\msM^C.
\end{equation} 
Since \eqref{EqOpLax} and \eqref{EqLax} are easily checked to correspond to each other under the adjunction, it then follows automatically that the ensuing equivalence between ${}_B\msM_B^A$ and ${}_A^C\msM^C$ is in fact a strong monoidal equivalence. 

Assume now again that $A$ is a CQG Hopf $*$-algebra and $B$ a right coideal $*$-subalgebra. Then indeed ${}_BA$ is faithfully flat \cite[Corollary 3.5]{Chi18} and $A^C$ is faithfully coflat (as $C$ is cosemisimple). There is then a one-to-one correspondence between Doi-Kopinnen modules and non-degenerate $\mcD(B,\mcI)$-modules through the correspondence 
\[
{}_B\msM^C \ni V \mapsto V \in {}_{\mcD(B,\mcI)}\msM, \qquad (b\omega) v = \tau(v_{(-1)},\omega)bv_{(0)},\qquad \omega \in \mcI,b\in B,v\in V. 
\]
It follows that the category of $\mcD(B,\mcI)$-modules obtains a uniquely defined monoidal structure through either its identification with  ${}_B\msM_B^A$ or ${}_A^C\msM^C$. 

Consider now however the C$^*$-category $\Rep(\mcD(B,\mcI))$ of non-degenerate $*$-representations of $\mcD(B,\mcI)$ on Hilbert spaces. We have a C$^*$-equivalence of this category with $\Rep_0(\mcD(B,\mcI))$, the C$^*$-category of non-degenerate $\mcD(B,\mcI)$-modules with a pre-Hilbert space structure making the module structure into a $*$-representation. Indeed, one easily checks that 
\[
\Rep(\mcD(B,\mcI)) \rightarrow \Rep_0(\mcD(B,\mcI)),\qquad \Hsp \mapsto \mcI\Hsp 
\] 
does the job, the quasi-inverse being given by completing a pre-Hilbert space into a Hilbert space (by Lemma \ref{LemRepIsBound}, we can extend the $*$-representation of $\mcD(B,\mcI)$ to this completion).

Hence, to obtain a monoidal structure on $\Rep(\mcD(B,\mcI))$, we only need to see how the Hilbert space structure can be transported via the equivalences in \eqref{EqEquiv}. Unfortunately, this seems not as straightforward as one would hope, and one  quickly runs into trouble with either questions concerning which C$^*$-completions to use, or how to invoke the appropriate modular structure to preserve the compatibility with the $*$-operation. In what follows, we will for this reason abandon the ${}_B\msM_B^A$-picture, and focus on the ${}_A^C\msM^C$-side, considered however from a dual point of view. 

Consider the von Neumann algebras $\msU_{\infty} \subseteq \msU$ and $\msI_{\infty}=\msI\cap \msU_{\infty}$ of uniformly bounded elements in the direct products \eqref{EqLinDual} and \eqref{EqIDirProd}. We then obtain normal maps
\[
\Delta: \msU_{\infty}\rightarrow \msU_{\infty}\overline{\otimes} \msU_{\infty},\qquad \Delta: \msI_{\infty}\rightarrow \msU_{\infty}\overline{\otimes} \msI_{\infty},
\]
where $\overline{\otimes}$ denotes the von Neumann algebraic tensor product (which just corresponds to bounded sequences in the algebraic completed product \eqref{EqAlgComplProd}). Hence $\msI_{\infty}$ can be seen as a left coideal von Neumann subalgebra of $\msU_{\infty}$. 

\begin{Def}
A right $W^*$-Hilbert $\msI_{\infty}$-bimodule $(\msE,\langle -,-\rangle_{\msI_{\infty}})$ will be called \emph{$\msU_{\infty}$-relative} if it is equipped with a normal comodule map $\Delta: \msE \rightarrow \msU_{\infty}\overline{\otimes} \msE$ such that 
\begin{equation}\label{EqIdEqInProd}
\Delta(x\xi y) = \Delta(x)\Delta(\xi)\Delta(y),\qquad \Delta(\langle \xi,\eta\rangle_{\msI_{\infty}}) = \langle \Delta(\xi),\Delta(\eta)\rangle_{\msU_{\infty}\overline{\otimes} \msI_{\infty}}, 
\end{equation}
where the latter $W^*$-valued inner product is determined by $\langle x\otimes \xi,y\otimes \eta\rangle_{\msU_{\infty}\overline{\otimes} \msI_{\infty}} = x^*y \otimes \langle \xi,\eta\rangle_{\msI_{\infty}}$. 
\end{Def}
Note that the equivariance conditions are meaningful as $\msI_{\infty}$ is a left coideal von Neumann subalgebra of $\msU_{\infty}$. Whenever convenient, we continue to use the Sweedler notation 
\[
\Delta(\xi) = \xi_{[-1]}\otimes \xi_{[0]}.
\]
We denote by ${}^{\msU_{\infty}}_{\msI_{\infty}}\msM_{\msI_{\infty}}$ the C$^*$-category of $\msU_{\infty}$-relative right W$^*$-Hilbert $\msI_{\infty}$-bimodules (with the obvious choice for morphisms). Note also for future reference that a right W$^*$-Hilbert $\msI_{\infty}$-module $\mcE$ can always be written in the form $\mcE\cong \Prod_{\beta}^b B(\Gsp_{\beta},V_{\beta})$ for Hilbert spaces $V_{\beta}$, where we take direct products which are uniformly bounded  in the operator norm. The inner product is then simply $\langle (x_{\beta})_{\beta},(y_{\beta})_{\beta}\rangle_{\msI_{\infty}} = (x_{\beta}^*y_{\beta})_{\beta} \in \msI_{\infty} = \Prod_{\beta}^b B(\Gsp_{\beta})$. 

If $\mcE$ is a $\msU_{\infty}$-relative right $W^*$-Hilbert $\msI_{\infty}$-bimodule, we can equip it with a left $A$-module structure by putting 
\[
a \blacktriangleright \xi= \tau(S_A^{-1}(a),\xi_{[-1]})\xi_{[0]},
\]
which is well-defined since $\tau(a,-)$ is a normal functional on $\msU_{\infty}$ for $a\in A$. Similarly, we equip $\msU_{\infty}$ with the left and right $A$-module 
\[
\tau(c,a\blacktriangleright x) = \tau(S_A^{-1}(a)c,x),\qquad \tau(c,x\blacktriangleleft a)= \tau(cS_A^{-1}(a),x),\qquad x\in \msU_{\infty},a,c\in A.
\]
Recall further the map $S_B^2$ introduced in \eqref{EqAntB}. 

\begin{Theorem}\label{TheoEquiCats}
There is an equivalence of C$^*$-categories
\[
\Rep(\mcD(B,\mcI)) \cong {}_{\msI_{\infty}}^{\msU_{\infty}}\msM_{\msI_{\infty}}.
\]
More precisely, let $\mcE \in {}_{\msI_{\infty}}^{\msU_{\infty}}\msM_{\msI_{\infty}}$, and consider $\Hsp = \mcE\Phi_C$ with the scalar product
\begin{equation}\label{EqHilbStruc}
\langle \xi,\eta\rangle = \tau(1_C,\langle\xi,\eta\rangle_{\msI_{\infty}}),\qquad \xi,\eta\in \Hsp. 
\end{equation}
Then $\Hsp$ is a Hilbert space with a non-degenerate $*$-representation of $\mcD(B,\mcI)$ through 
\[
\pi_{\Hsp}(b)\xi = S_A^2(S_B^{-2}(b))\blacktriangleright \xi,\qquad \pi_{\Hsp}(x)\xi = x\xi,\qquad b\in B,x\in \mcI,
\]
and the association $F_{\infty}: \mcE \mapsto \Hsp$ becomes an equivalence of tensor C$^*$-categories 
\[
F_{\infty}: {}_{\msI_{\infty}}^{\msU_{\infty}}\msM_{\msI_{\infty}} \rightarrow \Rep(\mcD(B,\mcI)).
\]
A quasi-inverse functor is given by 
\[
G_{\infty}: \Rep(\mcD(B,\mcI)) \rightarrow {}_{\msI_{\infty}}^{\msU_{\infty}}\msM_{\msI_{\infty}},\qquad \Hsp \mapsto \msU_{\infty}\overline{\square} \Hsp,
\]
where
\[ 
\msU_{\infty}\overline{\square} \Hsp = \{z \in \msU_{\infty} \overline{\otimes} \Hsp \mid \forall b\in B, (1\otimes \pi_{\Hsp}(b))z = z(\blacktriangleleft b \otimes \id)\}
\]
with the $\msU_{\infty}$-relative $\msI_{\infty}$-bimodule determined by ${}_{\bullet}^{\bullet}(\msU_{\infty})_{\bullet} \overline{\otimes} {}_{\bullet}\Hsp$ and with the Hilbert $\msI_{\infty}$-valued inner product inherited from the $\msU_{\infty}$-valued one on $\msU_{\infty} \overline{\otimes} \Hsp$ (i.e.\ the $\msU_{\infty}$-valued inner product on $\msU_{\infty}\overline{\square} \Hsp$ lands in $\msI_{\infty}$).

The adjoint equivalence between $F_{\infty}$ and $G_{\infty}$ is established through the unitary maps 
\[
\epsilon_{\Hsp} = \epsilon: (\msU_{\infty}\overline{\square} \Hsp)\Phi_C \rightarrow \Hsp,\qquad z \mapsto (\tau(1_A,-)\otimes \id)z
\]
and
\[
\eta_{\mcE}= \eta: \mcE \rightarrow \msU_{\infty} \overline{\square} (\mcE \Phi_C),\qquad \xi \mapsto \Delta(\xi)(1\otimes\Phi_C). 
\]
\end{Theorem} 

The proof of this theorem will be presented in the next section. 

Note that ${}^{\msU_{\infty}}_{\msI_{\infty}}\msM_{\msI_{\infty}}$ is naturally a tensor C$^*$-category by 
\[
\msE \widetilde{\otimes}\msF = \msE \otimes_{\msI_{\infty}} \msF,\qquad \Delta(\xi\otimes \eta) = \xi_{[-1]}\eta_{[-1]}\otimes \xi_{[0]}\otimes \eta_{[0]}. 
\]
There is no difficulty in interpreting the latter expression, since we can simply perform the multiplication separately with respect to each block of $\msU_{\infty}$. 

Using the equivalences $F_{\infty}$ and $G_{\infty}$, it follows in particular that we can endow $\Rep(\mcD(B,\mcI))$ with the structure of a tensor C$^*$-category.
\begin{Def} 
Let $\Gsp,\Hsp \in \Rep(\mcD(B,\mcI))$. The \emph{tensor product} of $\Gsp$ and $\Hsp$ is the Hilbert space
\[
\Gsp \boxtimes \Hsp := (\msU_{\infty} \overline{\square}\Hsp )\underset{\msI_{\infty}}{\overline{\otimes}} \Gsp \cong 
\left((\msU_{\infty} \overline{\square}\Hsp )\underset{\msI_{\infty}}{\overline{\otimes}} (\msU_{\infty} \overline{\square}\Gsp)\right)\Phi_C,
\]
where $\Gsp \boxtimes \Hsp \subseteq (\msU_{\infty} \overline{\otimes}\Hsp )\underset{\msI_{\infty}}{\overline{\otimes}} \Gsp $ is endowed with the $\mcD(B,\mcI)$-representation
\begin{multline*}
\pi_{\Gsp\boxtimes \Hsp}(x) \xi =  ((\id\otimes \pi_{\Hsp})\Delta(x)\otimes \id)\xi,\qquad \pi_{\Gsp\boxtimes \Hsp}(b) \xi =( b_{(2)} \blacktriangleright \otimes \id\otimes \pi_{\Gsp}(b_{(1)})) \xi, \\ x\in \mcI,b\in B,\xi \in \Gsp \boxtimes \Hsp,
\end{multline*}
and where the inner product is determined by 
\[
\langle \xi\otimes \eta, \xi'\otimes \eta' \rangle = \langle \eta, \langle \xi,\xi'\rangle_{\msI_{\infty}}\eta'\rangle,\qquad \xi,\xi'\in\msU_{\infty} \overline{\square}\Hsp,\eta,\eta'\in \Gsp. 
\]
\end{Def}
The extra flip between $\Gsp$ and $\Hsp$ is introduced to accommodate the convention that we take on the coideal $\mcD(B,\mcI)\subseteq \mcD(A,\mcU)$ the coproduct of $B$ but the opposite coproduct of $\mcI$. 

To end this section, we note that there is an intrinsic categorical interpretation of the C$^*$-category ${}_{\msI_{\infty}}^{\msU_{\infty}}\msM_{\msI_{\infty}}$. Let $\Rep(\mcU)$, resp.\ $\Rep(\mcI)$, be the C$^*$-categories of non-degenerate $*$-representations of $\mcU$, resp.\ $\mcI$. Then $\Rep(\mcU)$ is a tensor C$^*$-category, and $\Rep(\mcI)$ is a module C$^*$-category for $\Rep(\mcU)$: 
\[
\otimes:  \Rep(\mcU) \times \Rep(\mcU) \rightarrow \Rep(\mcU),\qquad  \boxtimes: \Rep(\mcU) \times \Rep(\mcI) \rightarrow \Rep(\mcI).
\]
\begin{Rem}
Denoting by $\Rep_{\fin}(\mcU)$, resp.\ $\Rep_{\fin}(\mcI)$ the C$^*$-categories of finite-dimensional non-degenerate $*$-representations, we can naturally identify $\Rep(\mcU)$ and $\Rep(\mcI)$ with the ind-completions of $\Rep_{\fin}(\mcU)$ and $\Rep_{\fin}(\mcI)$, since any non-degenerate $*$-representation of $\mcU$ or $\mcI$ is a direct sum of finite-dimensional $*$-representations, cf.\ \cite{NY16}. 
\end{Rem}

\begin{Theorem}\label{TheoIdentCat}
There is a strong monoidal unitary equivalence of tensor C$^*$-categories
\[
 {}_{\msI_{\infty}}^{\msU_{\infty}}\msM_{\msI_{\infty}}\cong \End_{\Rep(\mcU)}(\Rep(\mcI)).
\] 
\end{Theorem}

The right hand side is to be interpreted as follows: The C$^*$-category $\End_{\Rep(\mcU)}(\Rep(\mcI))$ is the C$^*$-category of couples $(F,u)$ with $F\in \End(\Rep(\mcI))$ a $*$-preserving cocontinuous endofunctor of $\Rep(\mcI)$ and $u$ a unitary natural transformation $-\boxtimes F(-) \cong F(-\boxtimes -)$ satisfying the obvious constraint under tensor products, i.e.\
\[
u_{V\boxtimes W,X} = u_{V,W\boxtimes X}(V\boxtimes u_{W,X}),\qquad V,W \in \Rep(\mcU),X\in \Rep(\mcI).
\] 
Morphisms in this category are  uniformly bounded natural transformations, endowed with their natural norm, and commuting with the associated  unitary natural transformations. 

\begin{proof}
Note that $\Rep(\mcI)$ equals the C$^*$-category of normal $\msI_{\infty}$-representations. It is then easily seen that any cocontinuous $*$-endofunctor $F$ of $\Rep(\mcI)$ must be implemented by a W$^*$-Hilbert $\msI_{\infty}$-bimodule $\msE = \msE_F$: for example one can take $\msE_F = \prod^{\bounded}_{\beta}\Hom(\msG_{\beta},F(\msG_{\beta}))$ if $\mcI = \oplus_{\beta} \End(\msG_{\beta})$. The assigment $\msE \mapsto F_{\msE} = \msE \otimes_{\msI_{\infty}} - $ is then an equivalence of tensor C$^*$-categories between ${}_{\msI_{\infty}}\msM_{\msI_{\infty}}$ and $\End(\Rep(\mcI))$ (the latter still consisting of \emph{cocontinuous} endofunctors). 

Assume now that $(F_{\msE},u) \in \End_{\Rep(\mcU)}(\Rep(\mcI))$. Then noting that $\boxtimes$ is implemented by the W$^*$-Hilbert bimodule ${}_{\Delta} (\msU_{\infty}\overline{\otimes}\msI_{\infty})_{\msU_{\infty}\overline{\otimes} \msI_{\infty}}$, we have that $-\boxtimes F_{\msE}(-)$ and $F_{\msE}(-\boxtimes-)$ are implemented respectively by ${}_{\Delta} (\msU_{\infty}\overline{\otimes}\msE_{\infty})_{\msU_{\infty}\overline{\otimes} \msI_{\infty}}$ and $\msE \otimes_{\msI} {}_{\Delta}(\msU_{\infty}\overline{\otimes} \msI_{\infty})_{\msU_{\infty}\overline{\otimes} \msI_{\infty}}$. Hence $u^*$ is determined as a unitary intertwiner
\[
u^*: \msE \otimes_{\msI} {}_{\Delta}(\msU_{\infty}\overline{\otimes} \msI_{\infty})_{\msU_{\infty}\overline{\otimes} \msI_{\infty}} \rightarrow {}_{\Delta} (\msU_{\infty}\overline{\otimes}\msE_{\infty})_{\msU_{\infty}\overline{\otimes} \msI_{\infty}},
\]
which must be of the form 
\[
u^*(\xi\otimes (x\otimes y)) = \Delta(\xi)(x\otimes y)
\]
for some $\Delta: \msE \rightarrow \msU_{\infty}\overline{\otimes} \msE$ satisfying \eqref{EqIdEqInProd}. As a unitary natural transformation, we then see that $u^*$ is given by
\[
u^*_{V,X}: \msE\otimes_{\msI_{\infty}}(V\boxtimes X) \rightarrow V \boxtimes (\msE \otimes_{\msI_{\infty}} X),\quad \xi\otimes (v\otimes w) \mapsto \xi_{[-1]}v\otimes (\xi_{[0]}\otimes w). 
\]
It is then also easily seen that the coherence condition for $u$ is equivalent to the coassociativity of $\Delta$. 

This provides us with an equivalence of C$^*$-categories ${}^{\msU_{\infty}}_{\msI_{\infty}}\msM_{\msI_{\infty}} \cong \End_{\Rep(\mcU)}(\Rep(\mcI))$. The monoidality follows by using the trivial identification maps 
\[
\msE\otimes_{\msI_{\infty}} (\msF\otimes_{\msI_{\infty}} X) = (F_{\msE} \circ F_{\msF})(X) \rightarrow F_{\msE\widetilde{\otimes} \msF}(X) = (\msE\otimes_{\msI_{\infty}} \msF)\otimes_{\msI_{\infty}} X, 
\]
\end{proof}

\begin{Rem}
We note that the use of relative Hilbert bimodules with respect to the quantum group itself as the coideal was already used to great effect in the theory of induction for locally compact quantum group representations, cf.\ \cite{Vae05}.
\end{Rem}

\subsection{Proof of Theorem \ref{TheoEquiCats}}

Given an $\msU_{\infty}$-comodule structure on a space $M$, we actually prefer to use the module structure 
\[
a\xi = \tau(a_{(1)},\delta_A^{-1/2}) a_{(2)}\blacktriangleright \xi =\tau(a_{(1)},\delta_A^{-1/2}) \tau(S_A^{-1}(a_{(2)}),\xi_{[-1]})\xi_{[0]},\qquad \xi\in M.
\]
Indeed, if then $\mcE$ is an $\msU_{\infty}$-relative right $W^*$-Hilbert $\msI_{\infty}$-bimodule, it is easily seen that $\mcE_0 = \mcI \mcE \mcI$ is endowed with a unique $A$-relative $C$-bicomodule structure with respect to the above $A$-module structure such that the left and right comodule structures $\delta_l,\delta_r$ of $C$ are determined by 
\[
(\id\otimes \tau(-,x))\delta_r(\xi) = x\xi,\qquad (\tau(-,x)\otimes \id)\delta_l(\xi) = \xi x,\qquad x\in \msI_{\infty}. 
\]
In particular, for  all $a\in A,x\in \msI_{\infty},\xi \in \mcE$ it holds that
\begin{equation}\label{EqCommIA}
x(a\xi) = \tau(a_{(2)},x_{(1)}) a_{(1)}(x_{(2)}\xi),\qquad (a\xi)x = \tau(a_{(1)},x_{(1)}) a_{(2)}(\xi x_{(2)}).
\end{equation}
Note that these are well-defined: the right hand sides are actually finite sums, since $A$ has finite support as functionals on $\msU_{\infty} \supseteq \msI_{\infty}$. 

Compatibility of the $A$-module structure and the $\msI_{\infty}$-valued inner product on $\mcE$ is governed by the identity 
\begin{equation}\label{EqCorrInnProd}
\tau(\sigma_A^{-1}(a)c,\langle \xi,\eta\rangle_{\msI_{\infty}}) = \tau(c,\langle a_{(1)}^*\xi,S_A(a_{(2)})\eta\rangle_{\msI_{\infty}}),\qquad a,c\in A,\xi,\eta\in \msI_{\infty}.
\end{equation}

Equivalent ways of writing \eqref{EqCorrInnProd} are
\begin{equation}\label{EqCorrInnProdAlt}
\tau(c,\langle a^*\xi,\eta\rangle_{\msI_{\infty}}) = \tau(\sigma_A^{-1}(a_{(1)})c,\langle \xi,a_{(2)}\eta\rangle_{\msI_{\infty}}),
\end{equation}
\begin{equation}\label{EqCorrInnProdAlt2}
 \tau(c,\langle\xi,S_A(a)\eta\rangle_{\msI_{\infty}}) = \tau(\sigma_A^{-1}(a_{(2)})c,\langle S_A(a_{(1)})^*\xi,\eta\rangle_{\msI_{\infty}}). 
\end{equation}

%We will split the proof of the equivalence of  $\Rep(\mcD(B,\mcI))$ and ${}_{\msI_{\infty}}^{\msU_{\infty}}\msM_{\msI_{\infty}}$ into different steps. 

Consider the complex one-parametergroup of automorphisms on $B$ given by 
\begin{equation}\label{EqDefKappa}
(\kappa_B)_z = \tau_{z}\circ (\sigma_A)_{-z}\circ (\sigma_B)_{2z},\qquad (\kappa_B)_z(\pi(w,v)) = \pi(\delta_B^{i\overline{z}}w,v),\qquad w\in \Phi_C\Hsp_{\alpha},v\in\Hsp_{\alpha}.
\end{equation}

The following proposition shows that the functor $F_{\infty}$ in Theorem \ref{TheoEquiCats} is well-defined.

\begin{Prop}
Let $\mcE \in {}_{\msI_{\infty}}^{\msU_{\infty}}\msM_{\msI_{\infty}}$, and consider $\Hsp = \mcE\Phi_C$ with the scalar product
\begin{equation}\label{EqHilbStruc}
\langle \xi,\eta\rangle = \tau(1_C,\langle\xi,\eta\rangle_{\msI_{\infty}}),\qquad \xi,\eta\in \Hsp. 
\end{equation}
Then $\Hsp$ is a Hilbert space, stable under the left $B$-module structure of $\mcE$, and 
\begin{equation}\label{EqStarRepDriDou}
\pi_{\Hsp}(b)\xi = (\kappa_B)_{i/2}(b)\xi,\qquad \pi_{\Hsp}(x)\xi = x\xi,\qquad b\in B,x\in \mcI_{\infty}
\end{equation}
defines a non-degenerate $*$-representation of $\mcD(B,\mcI)$. 
\end{Prop} 

It is easily checked that $(\kappa_B)_{i/2}(b)\xi = S_A^2(S_B^{-2}(b))\blacktriangleright \xi$, in agreement with Theorem \ref{TheoEquiCats}.

\begin{proof}
By the right $\msI_{\infty}$-linearity of $\langle -,-\rangle_{\msI_{\infty}}$ and the fact that $\Phi_C x = x\Phi_C = \tau(1_C,x)\Phi_C$ for $x\in \msI_{\infty}$, it is clear that 
\[
\langle \xi,\eta\rangle_{\msI_{\infty}} = \langle \xi,\eta\rangle \Phi_C,\qquad \xi,\eta\in \Hsp. 
\] 
Hence \eqref{EqHilbStruc} defines a pre-Hilbert space structure, and as $\mcE$ is complete, also $\Hsp$ must be complete, and hence a Hilbert space. 

Taking $b\in B,\xi \in \mcE$ we obtain from $b = \Phi_C(b_{(1)})b_{(2)}$, \eqref{EqCommIA} and \eqref{EqIdPhiCC} that
\begin{equation}\label{EqRepBH}
b(\xi\Phi_C) = (b_{(2)} \xi)\Phi_C(S_A(b_{(1)})-) = \Phi_C(S_A(b_{(1)}))(b_{(2)}\xi)\Phi_C,
\end{equation}
hence $\Hsp$ is a $B$-module upon restriction of the $A$-module structure on $\mcE$. Furthermore, for $a\in A,\eta\in \Hsp$ we have 
\[
(a\eta)\Phi_C =  a_{(2)}(\eta \Phi_C \Phi_C(a_{(1)}-)) = E_B(a)\eta,
\]
so from \eqref{EqCorrInnProdAlt2} we obtain for $\xi,\eta\in \Hsp$ and $b\in B$ that 
\begin{multline}\label{EqStarCompB}
\langle \xi,b\eta\rangle = \tau(S_A(b_{(1)}),\langle \sigma_A(S_A^2(b_{(2)}))^*\xi,\eta\rangle_{\msI_{\infty}}) = \Phi_C(S_A(b_{(1)}))\langle E_B(\sigma_A(S_A^2(b_{(2)}))^*)\xi,\eta\rangle \\= \langle \Phi_C(S_A^{-2}(b_{(1)}^*)) E_B(\sigma_A^{-1}(S_A^{-2}(b_{(2)}^*)))\xi,\eta\rangle,
\end{multline}
where in the last equality we used $\Phi_C^* = \Phi_C$. Now a direct computation shows that for $b = \pi(w,v) \in B$ with $w\in \Phi \Hsp_{\alpha}$ we have 
\[
\Phi_C(S_A^{-2}(b_{(1)})) E_B(\sigma_A^{-1}(S_A^{-2}(b_{(2)}))) = \pi(\Phi_C\delta_AS^2(\Phi_C)w,v) = \pi((\Phi_C\delta_A^{1/2}\Phi_C)^2w,v) = (\kappa_B)_i(b).
\]
It follows from \eqref{EqStarCompB} that \eqref{EqStarRepDriDou} indeed defines $*$-representations of $B$ and $\mcI$ respectively. As moreover $\kappa_B$ acts by convolution on the left, it is further clear that we still have from \eqref{EqCommIA} the commutation relation 
\[
\pi_{\Hsp}(x)\pi_{\Hsp}(b) =  \pi_{\Hsp}(b_{(1)}) \pi_{\Hsp}(x \lhd b_{(2)}),\qquad x\in \mcI,b\in B,
\]
so that we indeed obtain a $*$-representation of $\mcD(B,\mcI)$. It is clearly non-degenerate, as the $\mcI$-representation is obtained from restriction of a unital, normal $\msI_{\infty}$-representation. 
\end{proof}

Compatibility of $F_{\infty}$ with morphisms is immediate, hence $F_{\infty}$ defines a C$^*$-functor.
%From the above proposition, we obtain a C$^*$-functor
%\begin{equation}\label{EqFInfty}
%F_{\infty}: {}_{\msI_{\infty}}^{\msU_{\infty}}\msM_{\msI_{\infty}} \rightarrow \Rep(\mcD(B,\mcI)), \qquad \mcE \mapsto (\Hsp,\pi_{\Hsp})
%\end{equation}

We now construct the adjoint functor $G_{\infty}$ of Theorem \ref{TheoEquiCats}. As a preliminary remark, we note that any non-degenerate $*$-representation of $\mcI$ automatically extends to a normal $*$-representation of $\msI_{\infty}$. 

%an adjoint functor. 
%Let us also introduce the notation 
%\[
%(\tau_B)_{z}(\pi(w,v)) = \pi(\delta_B^{\frac{i \overline{z}}{2}}w,\delta_A^{\frac{iz}{2}}v),\qquad w\in \Phi_C\Hsp_{\alpha},v\in \Hsp_{\alpha}, 
%\] 
%and put $S_B^2 = (\tau_B)_{-i}$. 

%Also note that we view $\delta_C^{z}$ as an element in $\Phi_C \msU_{\infty}\Phi_C \subseteq \msU_{\infty}$. 

\begin{Prop}
Let $\Hsp$ be a non-degenerate $*$-representation of $\mcD(B,\mcI)$, and consider the $\msU_{\infty}$-relative right $\msU_{\infty}$-Hilbert W$^*$-module 
\[
\msU_{\infty} \overline{\otimes} \Hsp,\qquad \langle x\otimes \xi,y\otimes \eta\rangle_{\msU_{\infty}} = x^*y \langle \xi,\eta\rangle 
\]
in ${}_{\msI_{\infty}}^{\msU_{\infty}}\msM_{\msU_{\infty}}$, where $\msI_{\infty}$ acts via the coproduct and $\msU_{\infty}$ coacts by the coproduct in the first leg. Then 
\[
\langle \xi,\eta\rangle_{\msU_{\infty}} \in \msI_{\infty}, \qquad \xi,\eta\in \msU_{\infty}\overline{\square} \Hsp \subseteq \msU_{\infty}\overline{\otimes} \Hsp, 
\]
and by restricting the right $\msU_{\infty}$-structure to $\msI_{\infty}$ we obtain $\mcE = \msU_{\infty}\overline{\square} \Hsp \in {}_{\msI_{\infty}}^{\msU_{\infty}}\msM_{\msI_{\infty}}$.  
\end{Prop} 

\begin{proof}
For $x\in \msI_{\infty},b\in B$ and $z \in \msU_{\infty}\overline{\square} \Hsp$, we compute that 
\begin{eqnarray*}
&& \hspace{-2cm}  (z(x\otimes 1))(\blacktriangleleft b \otimes \id)
\\  &=& \tau(S_A^{-1}(b_{(1)}),x_{(2)}) (\id\otimes \tau(S_A^{-1}(b_{(2)}),-)\otimes \id)((\Delta\otimes \id)(z))(x_{(1)}\otimes 1)\\ 
& =&  (z(\blacktriangleleft b \otimes \id))(x\otimes 1)\\ 
&=& (1\otimes \pi_{\Hsp}(b))(z(x\otimes 1)),
\end{eqnarray*}
where in the third line we used that $\tau(S_A^{-1}(b),x) = \overline{\tau(b^*,x^*)} =  \varepsilon(b)\varepsilon(x)$ for all $b\in B$ and $x\in \msI_{\infty}$. It follows that  $(\msU_{\infty})_{\bullet}\overline{\square} \Hsp$ is a right $\msI_{\infty}$-module.

Let $c\in A$ and $b\in B$. Then we compute for $w,z\in \msU_{\infty}\overline{\square} \Hsp$ that 
\begin{eqnarray*}
\tau(cb,\langle w,z\rangle_{\msU_{\infty}}) 
&=& \tau(c,(\id\otimes \tau(S_A^{-1}(b_{(1)}^*),-)\otimes \id)(\Delta\otimes \id)w,(\id\otimes \tau(b_{(2)},-)\otimes \id)(\Delta\otimes \id)z\rangle_{\msU_{\infty}})\\
&=& \tau(c,\langle(\id\otimes \pi_{\Hsp}(b_{(1)}^*))w,(\id\otimes \tau(b_{(2)},-)\otimes \id)(\Delta\otimes \id)z\rangle_{\msU_{\infty}}) \\
&=& \tau(c,\langle w,(\id\otimes \tau(b_{(2)},-)\otimes \pi_{\Hsp}(b_{(1)}))(\Delta\otimes \id)z\rangle_{\msU_{\infty}}) \\
&=& \tau(c,\langle w,(\id\otimes \tau(b_{(2)},-)\otimes \tau(S_A^{-1}(b_{(1)}),-)\otimes \id)((\id\otimes \Delta)\Delta)z\rangle_{\msU_{\infty}})\\ 
&=& \varepsilon(b) \tau(c,\langle w,z\rangle_{\msU_{\infty}}).
\end{eqnarray*}
It follows that $\tau(-,\langle w,z\rangle_{\msU_{\infty}})$ vanishes on $AB_+$, and hence $\langle w,z\rangle_{\msU_{\infty}}$ defines an element of $\msI_{\infty}$. 

It follows that $(\msU_{\infty})_{\bullet}\overline{\square} \Hsp$ is a right W$^*$-Hilbert $\msI_{\infty}$-module. If then $x\in \msI_{\infty}$ and $z\in \msU_{\infty}\overline{\square} \Hsp$, we get for $b\in B$ that 
\begin{multline*}
(1\otimes \pi_{\Hsp}(b))(x_{(1)}\otimes \pi_{\Hsp}(x_{(2)}))z = \tau(S_A^{-1}(b_{(2)}),x_{(2)})(x_{(1)}\otimes \pi_{\Hsp}(x_{(3)})\pi_{\Hsp}(b_{(1)}))z \\  = \tau(S_A^{-1}(b_{(2)}),x_{(2)})(x_{(1)}\otimes \tau(S_A^{-1}(b_{(1)}),-)\otimes \pi_{\Hsp}(x_{(3)}))(\Delta\otimes \id)z \\= (\id\otimes \tau(S_A^{-1}(b_{(1)}),-)\otimes \id)(\Delta\otimes \id)((x_{(1)}\otimes \pi_{\Hsp}(x_{(2)}))z). 
\end{multline*}
Hence ${}_{\bullet}(\msU_{\infty})_{\bullet}\overline{\square} {}_{\bullet}\Hsp$ is a W$^*$-Hilbert $\msI_{\infty}$-bimodule. 

It is automatic from the definitions that the $\msU_{\infty}$-coaction via ${}^{\bullet}\msU_{\infty}\overline{\otimes} \Hsp$ preserves $\msU_{\infty} \overline{\square} \Hsp$, and that the latter thus becomes a $\msU_{\infty}$-relative W$^*$-Hilbert $\msI_{\infty}$-bimodule.
\end{proof}

Again compatibility of $G_{\infty}$ with the morphisms is immediate, so $G_{\infty}$ defines a C$^*$-functor.

%From the above proposition, we obtain a C$^*$-functor
%\begin{equation}\label{EqGInfty}
%G_{\infty}: \Rep(\mcD(B,\mcI)) \rightarrow {}_{\msI_{\infty}}^{\msU_{\infty}}\msM_{\msI_{\infty}}, \qquad \Hsp \mapsto \mcE = \msU_{\infty} \overline{\square} \Hsp.
%\end{equation}

Next, we consider the equivalence adjunction between $F_{\infty}$ and $G_{\infty}$. We first consider the counit.

\begin{Prop}\label{PropEqDefCounit}
Let $\Hsp \in \Rep(\mcD(B,\mcI))$ and let
\[
\epsilon_{\Hsp} = \epsilon: (\msU_{\infty}\overline{\square} \Hsp)\Phi_C \rightarrow \Hsp,\qquad z \mapsto (\varepsilon \otimes \id)z.
\]
Then $\epsilon$ is a unitary.
\end{Prop}

%\begin{Theorem}\label{TheoEquiRedCom}
%The functor $F_{\infty}$ and $G_{\infty}$ are quasi-inverse equivalences between the C$^*$-categories $\Rep(\mcD(B,\mcI))$ and ${}_{\msI_{\infty}}^{\msU_{\infty}}\msM_{\msI_{\infty}}$. 
%\end{Theorem}
\begin{proof}
 It follows directly from the definition of the inner product on the domain that $\epsilon$ is an isometry. To see that it is surjective, note first that we can define a map
\[
\delta: \Hsp \rightarrow \msU \widehat{\otimes} \Hsp = \Lin_{\C}(A, \Hsp),\qquad \xi \mapsto (a\mapsto \pi_{\Hsp}(F_B(S_A(a)))\xi),
\]
where $F_B$ was defined in \eqref{EqCondExp2}. We claim that $\delta(\xi) \in \msU_{\infty} \overline{\otimes}\Hsp$. Indeed, we compute for $a\in A$ and $\xi,\eta\in \Hsp$ that 
\begin{eqnarray*}
\tau(a,\delta(\xi)^*\delta(\eta)) &=& \langle \pi_{\Hsp}(F_B(a_{(1)})^*)\xi,\pi_{\Hsp}(F_B(S_A(a_{(2)})))\eta\rangle \\
&=& \langle \xi,\pi_{\Hsp}(E_B(a_{(1)}))\pi_{\Hsp}(F_B(S_A(a_{(2)})))\eta\rangle \\
&=&  \langle \xi,\pi_{\Hsp}(F_B(E_B(a_{(1)})S_A(a_{(2)})))\eta\rangle\\
&=& \Phi_C(a) \langle \xi,\eta\rangle.
\end{eqnarray*} 
Hence $\delta(\xi)^*\delta(\eta) =  \langle \xi,\eta\rangle \Phi_C$, and in particular $\delta(\eta)$ lies in $\msU_{\infty} \overline{\otimes}\Hsp$. The fact that $F_B$ is left $B$-linear leads immediately to $\delta(\xi) \in \msU_{\infty} \overline{\square}\Hsp$, and then $(\varepsilon \otimes \id)\delta(\xi) = \xi$. Hence $\epsilon$ is surjective. 
\end{proof}

We now consider the counit. 

\begin{Prop}\label{PropEqDefUnit}
Let $\mcE \in  {}_{\msI_{\infty}}^{\msU_{\infty}}\msM_{\msI_{\infty}}$. Let $\Hsp = \mcE\Phi_C$, and consider
\[
\eta_{\mcE}= \eta: \mcE \rightarrow \msU_{\infty} \overline{\otimes} \Hsp,\qquad \xi \mapsto \xi_{[-1]}\otimes \xi_{[0]}\Phi_C. 
\]
Then $\eta$ is a well-defined map into $\msU_{\infty} \overline{\square} \Hsp$, and in fact a unitary intertwiner.
\end{Prop}
\begin{proof}
Running through the definitions and using \eqref{EqRepBH}, we see that, for $\xi\in \mcE$ and $b\in B$,
\begin{eqnarray*}
\xi_{[-1]}\otimes \pi_{\Hsp}(b)(\xi_{[0]}\Phi_C) &=& \tau(b_{(1)},\Phi_C\delta_A^{1/2}S(\Phi_C)\delta_A^{-1/2})\tau(S_A^{-1}(b_{(2)}),\xi_{[-1]})\xi_{[-2]}\otimes \xi_{[0]}\Phi_C \\
&=&  \tau(b_{(1)},\Phi_C S_A^{-1}(\Phi_C))\tau(S_A^{-1}(b_{(2)}),\xi_{[-1]})\xi_{[-2]}\otimes \xi_{[0]}\Phi_C\\
&=& \tau(b_{(1)},\Phi_C)\tau(S_A^{-1}(b_{(2)}),\xi_{[-1]})\xi_{[-2]}\otimes \xi_{[0]}\Phi_C \\
&=& \tau(S_A^{-1}(b),\xi_{[-1]})\xi_{[-2]}\otimes \xi_{[0]}\Phi_C,
\end{eqnarray*}
proving that in fact $\xi_{[-1]}\otimes \xi_{[0]}\Phi_C \in \msU_{\infty} \overline{\square} \Hsp$. 

It is immediate then that $\eta$ is an $\msU_{\infty}$-equivariant $\msI_{\infty}$-bimodule map preserving the $\msI_{\infty}$-valued inner product. It remains again to see that $\eta$ is surjective.  If $\eta$ were not surjective, then by the existence of orthogonal complements in W$^*$-Hilbert $\msI_{\infty}$-modules there would exist a non-zero $z\in \msU_{\infty} \overline{\square} \Hsp$ with $z^*\Delta(\xi)(1\otimes \Phi_C)=0$ for all $\xi \in \mcE$, hence also $z^*\Delta(\xi)(x\otimes \Phi_C)= 0$ for all $\xi\in \mcE\Phi_C,x\in \msU_{\infty}$. However, it is not hard to show that the $\Delta(\xi)(x\otimes \Phi_C)$ are $\sigma$-weakly closed in $\msU_{\infty} \overline{\otimes} \Hsp$, leading to $z=0$, a contradiction.
\end{proof} 

It is further straightforward to check that the maps $\epsilon_{\Hsp},\eta_{\mcE}$ in Proposition \ref{PropEqDefCounit} and Proposition \ref{PropEqDefUnit} indeed define an adjoint pairing between $F_{\infty}$ and $G_{\infty}$, and hence turn $G_{\infty}$ into a quasi-inverse for $F_{\infty}$, finishing the proof of Theorem \ref{TheoEquiCats}.

\section{Relatively invariant functionals on $\mcU_q(\mfsl(2,\R)_t)$} 

Let $0<q<1$, and let $A=\mcO_q(SU(2))$ be the universal Hopf algebra generated by $\alpha,\beta,\gamma,\delta$ such that 
\begin{subequations}\label{EqDefRelSUq}
\begin{gather}
\alpha\beta = q\beta \alpha,\quad \alpha \gamma = q\gamma \alpha,\quad \beta\gamma = \gamma \beta,\quad \beta\delta =q\delta\beta,\quad \gamma \delta = q\delta\gamma,\\
\alpha\delta - q\beta\gamma = 1,\qquad \delta \alpha -q^{-1}\gamma \beta = 1. 
\end{gather}
\end{subequations}
It is well-known that $\mcO_q(SU(2))$ is a CQG Hopf $*$-algebra by the $*$-structure 
\[
\begin{pmatrix} \alpha^* & \gamma^* \\ \beta^* & \delta^* \end{pmatrix} = \begin{pmatrix} \delta & -q^{-1}\beta \\ -q\gamma & \alpha\end{pmatrix}
\]
and the coproduct making $U = \begin{pmatrix} \alpha & \beta \\ \gamma & \delta \end{pmatrix}$ a corepresentation \cite{Wor87a,Wor87b}. Moreover, it is non-degenerately paired with the Hopf $*$-algebra $U=U_q(\mfsu(2))$ generated by elements $e,f,k^{\pm 1}$ with universal relations 
\[
ke = q^2ek,\qquad kf = q^{-2}fk,\qquad ef-fe = \frac{k-k^{-1}}{q-q^{-1}}
\]
and Hopf $*$-algebra structure determined by 
\[
e^* = fk,\qquad f^* = k^{-1}e,\qquad k^* =k,
\]
\[
\Delta(e) = e\otimes 1 + k\otimes e,\qquad \Delta(f) = f\otimes k^{-1}+ 1\otimes f,\qquad \Delta(k) = k\otimes k.
\]
The admissible $*$-representations of $U_q(\mfsu(2))$ are labeled by half-integers. Concretely, we write $(V_{n/2},\pi_{n/2})$ for the $n+1$-dimensional unitary representation of $U_q(\mfsu(2))$, with orthonormal basis $\xi_{0},\xi_1,\ldots,\xi_n$ and action
\[
\pi_{n/2}(k)\xi_p =  q^{n-2p} \xi_p,
\]
\[
(q^{-1}-q)\pi_{n/2}(e)\xi_p = \sqrt{(q^{-n+p-1}-q^{n-p+1})(q^{-p}-q^p)}q^{\frac{n}{2}-p+1}\xi_{p-1},
\]
\[
(q^{-1}-q)\pi_{n/2}(f)\xi_p = \sqrt{(q^{-n+p}-q^{n-p})(q^{-p-1}-q^{p+1})} q^{-\frac{n}{2}+p} \xi_{p+1}. 
\]
We denote by $U_{n/2} \in \End(\C^{n+1})\otimes \mcO_q(SU(2))$ the associated unitary corepresentation of $\mcO_q(SU(2))$, and for $\xi,\eta\in \C^{n+1}$ by $U_{n/2}(\xi,\eta)\in \mcO_q(SU(2))$  its matrix coefficients.

Write $\msU = \msU_q(\mfsu(2)) = \Lin(\mcO_q(SU(2)))$ for the dual convolution $*$-algebra. It is well-known that we can express the modular element $\delta_A \in \msU_q(\mfsu(2))$ for $\mcO_q(SU(2))$ by 
\begin{equation}\label{Eqdelk}
\delta_A^{1/2} = k. 
\end{equation}

Consider now for $t\in \R$ the element
\begin{equation}\label{EqStarCoidU}
B_t = q^{-1/2}(e-fk) - i(q-q^{-1})^{-1}tk.  
\end{equation}
Then $B_t^* = -B_t$ and
\[
\Delta(B_t) = q^{-1/2}(e-fk)\otimes 1 + k\otimes B_t,
\]
so that $B_t$ generates a $*$-invariant left coideal $I = U_q(\mfk_t)$ of $U$. The stabilizer dual $I^{\perp} = B = \mcO_q(S_t^2)$ can be identified with one of the \emph{Podle\'{s} spheres} \cite{Pod87}. We recall from \cite[Theorem 4.3]{Koo93} (see also \cite[Theorem 2.1]{DCDz21}) that the spectrum of $\pi_{n/2}(iB_{t})$ equals the set
\[
\Spec(\pi_{n/2}(iB_{t})) =\left\{[a+n-2p]\mid p=0,1,\ldots,n\right\}
\]
when $t =q^{a}-q^{-a}$ and where $[x] = \frac{q^x-q^{-x}}{q-q^{-1}}$ for $x\in \R$. It hence follows that $\mcI = \mcU_q(\mfk_t)$ will be isomorphic to the $*$-algebra of functions with finite support on the set $\{[a+n]\mid n \in \Z\}$. We denote by $e_{[a+n]}$ the Dirac function at $[a+n]$. 

\begin{Theorem}\label{TheoInvInt}
There exists a $\delta_A^{-1/2}$-invariant integral $\psi_t: \mcU_q(\mfk_t)\rightarrow \C$, concretely determined by
\[
\psi_t(e_{[a+n]}) =  \frac{q^{a+n}+q^{-a-n}}{q^a + q^{-a}},\qquad \forall n\in \Z. 
\]
\end{Theorem} 
\begin{proof}
To see that a $\delta_A^{-1/2}$-invariant integral exists, it is by Theorem \ref{TheoInvFuncIsChar} (and the identity $k =\delta_A^{1/2}$) sufficient to verify that 
\begin{equation}\label{Eqksig}
\tau(b,k^{-1}) = \varepsilon(\sigma_A\sigma_B^{-1}(b)),\qquad b\in \mcO_q(S_t^2).
\end{equation}
Now $\mcO_q(S_t^2)$ is generated by its spin one spectral subspace, which is spanned by the $U_1(\xi^{(t;1)},\eta)$ with $\xi^{(t;1)} = \begin{pmatrix} q^{-1/2} \\ (q+q^{-1})^{-1/2}it \\ q^{1/2}\end{pmatrix}$. Noting that the $V_{n/2}$ are self-dual by the unitary transformation $u_n: \overline{\xi}_p \mapsto (-1)^p\xi_{n-p}$, it follows that $R(\Phi_C)$ is the orthogonal projection onto 
\[
u_2 \C\overline{\xi^{(t;1)}} =  \C \begin{pmatrix} q^{1/2} \\ (q+q^{-1})^{-1/2}it \\ q^{-1/2}\end{pmatrix},
\]
and hence $\delta_A^{-1/4}R(\Phi_C)\delta_A^{-1/4}$ leaves $\xi^{(t;1)}$ invariant. From \eqref{EqModAutB} and  \eqref{EqDelasFun} we then find that \eqref{Eqksig} holds. 

Write now $\psi_t(e_{[a+n]}) = \mu_{[a+n]}$. It then follows from \eqref{EqSumNorm} (with $c_{\beta}=1$) that 
\[
\mu_{[a+n]} = \langle \xi,k\xi\rangle
\]
for $\xi$ any normalized $[a+n]$-eigenvector of $iB_t$. Now it is easily computed that for $n\geq 0$ we have in $V_{n/2}$ the $[a+n]$-eigenvector
\[
\xi_{n/2}^+ = \sum_{p=0}^n (-iq^{-(a+n)})^p \sqrt{\frac{(q^{2n};q^{-2})_p}{(-1)^p (q^{-2};q^{-2})_p}}\xi_p,
\]
where we use the Pochhammer symbol $(x;q^2)_n = (1-x)(1-q^2x)\ldots (1-q^{2n-2}x)$. By the $q$-binomial theorem, we compute the norm of $\xi_{n/2}^+$ to be
\[
\|\xi_{n/2}^+\|^2 = (-q^{-2a};q^{-2})_n.
\]
Another application of the $q$-binomial theorem then leads to 
\[
\mu_{[a+n]} = \frac{\langle \xi_{n/2}^+ ,k\xi_{n/2}^+\rangle}{\|\xi_{n/2}^+\|^2} = \frac{q^{a+n}+q^{-a-n}}{q^a + q^{-a}}.
\]
The proof for $n<0$ is similar. 
\end{proof}

\begin{Rem}
The previous theorem can also be proven by observing that $\mcO_q(S_t^2)\subseteq \mcO_q(SU(2))$ is a co-Gelfand pair. In fact, this is true also in higher rank for any coideal obtained from quantum symmetric pair coideals, using \cite[Theorem 4.2]{Let03} (note that the construction there needs to be tweaked a little to obtain $*$-invariance, see \cite[Appendix B]{DCNTY19}). As we do not develop this result further in this paper, we refrain from spelling out further details.  
\end{Rem}

By the theory developed in Section \ref{SecInfFuncDD}, we can consider the Drinfeld double coideal $\mcD(\mcO_q(S_t^2),\mcU_q(\mfk_t))$. As motivated in \cite{DCDz21} and in the second section of this paper (see also \cite{Pus93,PW94,BR99} for similar motivations in the context of realifications of complex quantum groups), we can think of this $*$-algebra as a $q$-deformation of the convolution $*$-algebra of compact support continuous functions on (a conjugated version of) $SL(2,\R)$. To be in accordance with earlier notation, we will the write this $*$-algebra as
\[
\mcU_q(\mfsl(2,\R)_t) =  \mcD(\mcO_q(S_t^2),\mcU_q(\mfk_t)).
\]

From Corollary \ref{CorDrinfDoubTrace}, we obtain the following.

\begin{Cor}
Let $\Phi_A$ be the Haar state on $\mcO_q(SU(2))$. Then
\[
\varphi_{t}: \mcU_q(\mfsl(2,\R)_t) \rightarrow \C,\qquad e_{[a+n]}b \mapsto \frac{q^{a+n}+q^{-a-n}}{q^a + q^{-a}}\Phi(b)
\]
is a positive, tracial, $k^{-1}$-invariant functional. 
\end{Cor}

Let $L^2_q(S_t^2)$ be the GNS-space of $\mcO_q(S_t^2)$ with respect to $\Phi$, and $L^2_q(K_t)$ the GNS-space of $\mcU_q(\mfk_t)$ with respect to $\psi_t$. Then we can form the Hilbert space 
\[
L^2_q(SL(2,\R)_t) = L^2_q(S_t^2) \otimes L^2_q(K_t)
\] 
and endow it with the regular $*$-representation of $\mcU_q(\mfsl(2,\R)_t)$ as in Definition \ref{DefRegRep}. 

We have thus obtained a non-trivial $q$-deformation of the regular representation of $SL(2,\R)$. The stage is now set to understand the decomposition of this regular representation $\pi_{\reg}$ in terms of the irreducible representations of $U_q(\mfsl(2,\R)_t)$ classified in \cite{DCDz21}, and to obtain an associated Plancherel formula for the relatively invariant functional on $\mcU_q(\mfsl(2,\R)_t)$. This detailed study of $\pi_{\reg}$ will be left for a future occasion.

\end{document}